\documentclass[11 pt]{amsart}

\usepackage[utf8]{inputenc}
\usepackage[margin=2.5cm]{geometry}
\usepackage{amsfonts,graphicx,enumerate}
\usepackage{amssymb, mathrsfs}
\usepackage{amsmath}
\usepackage{latexsym}
\usepackage{mathrsfs}
\usepackage{amsthm}
\usepackage{verbatim}
\usepackage{amscd}
\usepackage{mathtools}
\usepackage[pagebackref]{hyperref}
\hypersetup{bookmarksdepth=2}

\usepackage{import}
\usepackage{xifthen}
\usepackage{pdfpages}
\usepackage{transparent}
\usepackage{subcaption}

\usepackage{comment}
\usepackage{framed}
\usepackage{color}
\definecolor{shadecolor}{gray}{0.875}

\newtheorem{thrm}{Theorem}[section]
\newtheorem{lem}[thrm]{Lemma}
\newtheorem{cor}[thrm]{Corollary}
\newtheorem{prop}[thrm]{Proposition}

\theoremstyle{definition}
\newtheorem{defn}[thrm]{Definition}

\newtheorem{ex}[thrm]{Example}
\newtheorem{rmk}[thrm]{Remark}
\newtheorem{ques}[thrm]{Question}

\newtheorem{conv}[thrm]{Convention}

\DeclareMathOperator{\Sym}{Sym}

\DeclareMathOperator{\Eff}{\overline{Eff}}

\DeclareMathOperator{\Nef}{{Nef}}

\DeclareMathOperator{\ch}{ch}

\DeclareMathOperator{\td}{Td\,}

\DeclareMathOperator{\Spec}{Spec}

\let\cal\mathcal

\let\bb\mathbb

\DeclareMathOperator{\End}{{\cal E}nd\,}

\DeclareMathOperator{\sgn}{sgn\,}

\usepackage[all]{xy}

\newcommand{\factor}[2]{\left. \raise 1pt\hbox{\ensuremath{#1}} \right/
        \hskip -2pt\raise -3pt\hbox{\ensuremath{#2}}}

\numberwithin{equation}{thrm}

\begin{document}

\title{Positivity vs. slope semistability for bundles with vanishing discriminant}
\author{Mihai Fulger}
\address{Department of Mathematics, University of Connecticut, Storrs, CT 06269-1009, USA}
\address{Institute of Mathematics of the Romanian Academy, P. O. Box 1-764, RO-014700,
Bucharest, Romania}
\email{mihai.fulger@uconn.edu}
\author{Adrian Langer}
\address{Institute of Mathematics, University of Warsaw,
	ul.\ Banacha 2, 02-097 Warszawa, Poland}
\email{alan@mimuw.edu.pl}
\thanks{The first author was partially supported by the Simons Foundation
	Collaboration Grant 579353. The second author was partially supported by Polish National Centre (NCN) contract numbers 2018/29/B/ST1/01232 and 2021/41/B/ST1/03741.}

\maketitle


\begin{abstract}It is known that a strongly slope semistable bundle with vanishing discriminant is nef if and only if its determinant is nef. We give an algebraic proof of this result in all characteristics and generalize it to arbitrary proper schemes. 
We also address a question of S. Misra.
\end{abstract}

\section{Introduction}

Let $\cal E$ be a vector bundle of rank $r$ on a smooth projective variety $X$ defined over an algebraically closed field. Inspired by the case of line bundles, one might hope that the positivity of $\cal E$ is determined by the positivity of its characteristic classes. The bundle $\cal O_{\bb P^1}(n)\oplus\cal O_{\bb P^1}(-n)$ is an easy counterexample. To rectify this, one adds stability assumptions on $\cal E$.

Let $H$ be an ample polarization on $X$. In characteristic zero, say that $\cal E$ is \emph{strongly slope semistable with respect to $H$} if it is slope semistable in the usual sense. In positive characteristic, say that $\cal E$ is \emph{strongly slope semistable with respect to $H$} if $\cal E$ and all its iterated Frobenius pullbacks are slope semistable. 

On curves the polarization is irrelevant and we just say that $\cal E$ is \emph{strongly semistable}. Here the connection between positivity and strong semistability is well known by work of Hartshorne \cite{har71}, Barton \cite{Barton71}, and Miyaoka \cite{Miyaoka}. A strongly semistable bundle $\cal E$ on a curve is ample (or just nef) if and only if $\int_X c_1(\cal E)>0$ (resp.~ $\geq 0$). On surfaces, even on $\bb P^2$, it is not sufficient. See Example \ref{ex:notnef}.
Furthermore, on curves $\cal E$ is strongly semistable if and only if the twisted normalized bundle $\cal E\langle-\frac 1r\det\cal E\rangle$ is nef (equivalently $\cal E{\rm nd}\, \cal E$ is nef). 

A link in codimension two between semistability and positivity comes from the famous Bogomolov inequality. The discriminant of $\cal E$ is $\Delta(\cal E)=2rc_2(\cal E)-(r-1)c_1^2(\cal E)$. Assume that $\cal E$ is strongly slope semistable with respect to some ample polarization. The classical form of the inequality states that if $X$ is a surface, then the degree of the discriminant of $\cal E$
is non-negative. When $X$ has arbitrary dimension, the Mehta--Ramanathan theorem \cite{MehtaRamanathan} implies that $\Delta(\cal E)$ has nonnegative degree with respect to any polarization. 

Note that $\Delta(\cal E)=0$ on curves. This suggests a close connection between (strong) semistability and positivity for vector bundles that are extremal with respect to the Bogomolov inequality, meaning that $\Delta(\cal E)$ is numerically trivial.
We have the following known results:

\begin{thrm}\label{thm:main}
	Let $X$ be a smooth projective variety of dimension $n$, and let $H$ be an ample class on $X$. Let $\cal E$ be a reflexive sheaf of rank $r$ on $X$. The following are equivalent:
	\begin{enumerate}
		\item $\cal E$ is strongly slope semistable with respect to $H$, and $\Delta(\cal E)\cdot H^{n-2}=0$.
		\item $\cal E$ is locally free and $\cal E{\rm nd}\, \cal E$ is nef.
		\item $\cal E$ is universally semistable (see below).
	\end{enumerate} 
In particular, if $\cal E$ is strongly slope semistable with respect to $H$, and $\Delta(\cal E)\cdot H^{n-2}=0$, then $\cal E$ is a nef (resp.~ample) vector bundle if and only if $\det\cal E$ is nef (resp.~ample).
\end{thrm}

In the above theorem we say that $\cal E$ is \emph{universally semistable} if for every morphism $f:Y\to X$ from any projective manifold $Y$, the pullback $f^*\cal E$ is slope semistable with respect to any ample polarization on $Y$. It is sufficient to check this condition for morphisms from curves.

When $\det\cal E$ is numerically trivial, condition $(2)$ is simply that the vector bundle $\cal E$ is nef, equivalently $\cal E$ is \emph{numerically flat} ($\cal E$ and $\cal E^{\vee}$ are nef). Theorem \ref{thm:main} takes the following form:

\begin{thrm}\label{thm:main2}
	Let $X$ be a smooth projective variety of dimension $n$ over an algebraically closed field, and let $H$ be an ample polarization. Let $\cal E$ be a reflexive sheaf on $X$. The following are equivalent:
	\begin{enumerate}
		\item $\cal E$ is strongly slope semistable with respect to $H$, and $c_1(\cal E)\cdot H^{n-1}=\ch_2(\cal E)\cdot H^{n-2}=0$.
		\item $\cal E$ is universally semistable and all the Chern classes $c_i(\cal E)$ are numerically trivial for $i\geq 1$.
		\item $\cal E$ is locally free and numerically flat.
	\end{enumerate}
In particular, if $\cal E$ is a strongly slope semistable with respect to $H$ with $c_1(\cal E)$ numerically trivial, then $\cal E$ is nef if and only if $\Delta(\cal E)$ is numerically trivial.
\end{thrm}

\noindent The result fails if $\cal E$ is only assumed to be torsion free. It also fails if the condition $\ch_2(\cal E)\cdot H^{n-2}=0$ in (1) is replaced by $c_2(\cal E)\cdot H^{n-2}=0$.

The two theorems are essentially equivalent. As mentioned before, they are not new. In characteristic zero, Theorem \ref{thm:main} is proved by \cite{NakayamaNormalized}. It is proved by \cite{GKP16} for more general polarizations by movable curves on varieties with mild singularities. It is proved by \cite{BBprincipal} for principal $G$-bundles, and by \cite{BHHiggs} as part of their study of Higgs bundles. 
Note that for Higgs bundles with nontrivial Higgs field it is not known whether an analogue of $(3)\Rightarrow (1)$ holds (this is known as Bruzzo's conjecture). See \cite{BBG19}. The proofs given in the references above have important transcendental components coming from \cite{Simpson} or \cite{DPS94}. The last statement of Theorem \ref{thm:main} has also been observed by \cite{MisraRayAmple}.
In positive characteristic, a version of the result is proved algebraically by the second named author in \cite{Langerflat} making crucial use of the Frobenius morphism. In characteristic zero, \cite[Corollary 4.10]{Langer19} gives an algebraic proof of the implication $(1)\Rightarrow (3)$ for the more general case of Higgs bundles by reduction to positive characteristic. See also \cite[Theorem 12]{Langer15}.
In characteristic zero, Theorem \ref{thm:main2} is proved in \cite[Theorem 2]{Simpson}. A positive characteristic version appears in \cite[Proposition 5.1]{Langerflat}. 

\medskip

A conjecture of Bloch predicts that the Chern classes of numerically flat bundles vanish in $B^*(X)\otimes \mathbb Q$, where $B^m(X)$ is the group of codimension $m$ cycles modulo algebraic equivalence (see \cite[Conjecture 3.2]{La-Chern}). This would strengthen Theorem \ref{thm:main2}. However, this is known only if $k$ has positive characteristic (see \cite[Corollary 3.7]{La-Chern}). In this case the result holds also for general proper schemes (see Corollary \ref{cor-Bloch}).
Analogously, one can formulate a similar conjecture and a result for bundles in Theorem \ref{thm:main} (see \cite[Proposition 3.6]{La-Chern}).
If $k=\mathbb C$ one can prove only a weaker version of the vanishing of Chern classes in the rational cohomology $H^{2*}(X^{\rm an}, \mathbb Q)$, where $X^{\rm an}$ is the complex manifold underlying variety $X$. This result would follow from Bloch's conjecture by using the cycle map $B^*(X)\otimes \mathbb Q\to H^{2*}(X^{\rm an}, \mathbb Q)$.

\subsection{Main results}
We give algebraic proofs for Theorems \ref{thm:main} and \ref{thm:main2} that are characteristic free.  We avoid the application of the non-abelian Hodge theorem \cite[Corollary 1.3]{Simpson} or the reduction to positive characteristic techniques of \cite{Langerflat, Langer19}. Moreover, we generalize both theorems to general proper schemes over an algebraically closed field (see Theorem \ref{main:proper-num-flat} and Corollary \ref{main:proper-universally-ss}).

\subsection{On a question of S.~Misra}
On curves $C$ it follows from \cite{Miyaoka} that $\cal E$ is strongly semistable if and only if every effective divisor on $\bb P_C(\cal E)$ is nef.
The first named author proved in \cite{ful11} a generalization of this for cycles of arbitrary (co)dimension in $\bb P_C(\cal E)$.
It is interesting to see to what extent does Miyaoka's result carry over to a projective manifold $X$ of arbitrary dimension. The equality $\Eff(X)=\Nef(X)$ of cones of divisors is a necessary condition which held trivially on curves. 
Here one has the following result:

\begin{thrm}\label{thm:misra}Let $X$ be a smooth projective variety such that $\Eff(X)=\Nef(X)$.
	Let $\cal E$ be a strongly slope semistable bundle with respect to some ample polarization of $X$ and assume that $\Delta(\cal E)\equiv0$.
	Then $\Eff(\bb P(\cal E))=\Nef(\bb P(\cal E))$.
\end{thrm}

In the characteristic zero case this result was proven by S. Misra in \cite[Theorem 1.2]{Misra21} as an  application of Theorem \ref{thm:main}. A similar proof works also in an arbitrary characteristic.
 Misra \cite[Question 3.11]{Misra21} also asks about a possible converse to this result.

\begin{ques}\label{ques:main}
Let $X$ be a smooth projective variety.	Let $\cal E$ be a vector bundle on $X$ such that every effective divisor on $\bb P(\cal E)$ is nef.
Is it true that $\cal E$ is slope semistable with respect to some (any) polarization of $X$ and $\Delta(\cal E)\equiv 0$? Or equivalently, is $\End \cal E$ numerically flat?
\end{ques}

\noindent For every $n\geq 2$, the tangent bundle $T_{\bb P^n}$ is a counterexample to the current phrasing of the question. In Example \ref{ex:unstableexample} we even give a slope unstable counterexample. However \cite{Misra21} also observes that under the hypothesis of Theorem \ref{thm:misra} the effective and nef cones of divisors coincide on $\bb P(\Sym^m\cal E)$ for all $m\geq 0$.
This motivates the following positive answer to a version of Question \ref{ques:main} in arbitrary characteristic.

\begin{thrm}\label{thm:plethysm}
	Let $X$ be a smooth  projective variety. Let $\cal E$ be a vector bundle on $X$ such that every effective divisor on $\bb P(\Sym^m\cal E)$ is nef for all $m\geq 0$. Then $\cal E$ is strongly slope semistable with respect to any polarization $H$ and $\Delta(\cal E)\equiv 0$.
\end{thrm}

\noindent Our proof shows that in fact existence of only one positive \emph{even} value $2m$ for which the cones $\Eff(\bb P(\Sym^{2m}\cal E))$ and $\Nef(\bb P(\Sym^{2m}\cal E))$ coincide is sufficient. 
The key idea is a result on plethysms which guarantees that the line bundle $(\det\cal E)^{\otimes 2m}$ is a subbundle of $\Sym^r\Sym^{2m}\cal E$. In characteristic zero, this also follows from \cite{BCI}. 
One can also see that equality $\Eff(\bb P(\End \cal E))=\Nef(\bb P(\End \cal E))$ implies that $\End \cal E$ is numerically flat, which provides a satisfactory answer to the original question.

\subsection{Acknowledgments}
The authors would like to thank  S.~Misra, D. S.~Nagaraj and J. Weyman for useful comments, suggestions, and references. 

\section{Preliminaries}

Let  $X$ be a connected proper scheme over an algebraically closed field $k$
and $\cal E$ be a vector bundle on $X$. Note that terms vector bundle and locally free sheaf are used interchangeably (since $X$ is connected, a locally free sheaf has the same rank at every point of $X$).

\subsection{Positive bundles}

Let $\bb P(\cal E)=\bb P_X(\cal E)={\rm Proj}_{\cal O_X}\Sym^{\bullet}\cal E$ with natural bundle map $\pi:\bb P(\cal E)\to X$, and let $\xi=c_1(\cal O_{\bb P(\cal E)}(1))$. 
We say that $\cal E$ is \emph{ample} (resp.~\emph{nef}) on $X$, if $\xi$ is ample (resp.~nef) on $\bb P(\cal E)$. The definition also makes sense for coherent sheaves.

\subsection{Numerically trivial Chern classes}

In the following we write $A_*(X)$ for the Chow group of rational equivalence classes on $X$ 
and $A^*(X)=A^*(X\stackrel{\rm id}{\longrightarrow} X)$ for the operational Chow ring, i.e., the group of bivariant rational equivalence classes (see \cite[Chapter 17]{fulton84}).
The Chern classes  of $\cal E$ are  operations on the Chow group (see \cite[Chapter 3]{fulton84}) so they are elements of $A^*(X)$. Following \cite[Chapter 19]{fulton84}, we say that  $c_j(\cal E)$ is \emph{numerically trivial} if for every proper closed subscheme $Y\subseteq X$ of dimension $j$ we have $\int_Yc_j(\cal E)=\int _X c_j (\cal E)\cap [Y]=0$, where $\int_X: A_0(X)\to\bb Z$ denotes the natural degree map. Then we write $c_j (\cal E)\equiv 0$. By additivity it is sufficient to check the condition $\int _X c_j (\cal E)\cap [Y]=0$ for all $Y$ that are irreducible and reduced. Similarly, we can define numerical triviality for any polynomial in Chern classes of $\cal E$, or even in Chern classes of finitely many bundles.

\subsection{Positive polynomials}
Consider $n\geq 1$ and grade $\bb Q[c_1,\ldots,c_n]$ so that $\deg c_i=i$.
Let $P(c_1,\ldots,c_n)$ be a weighted homogeneous polynomial of degree $n$.
Fulton and Lazarsfeld proved in \cite{fl83} that $\int _X P(\cal E)\geq 0$ for every $n$-dimensional variety $X$ and every nef vector bundle $\cal E$ on $X$ if and only if $P$ is a linear combination with nonnegative coefficients of Schur polynomials of degree $n$. We call such polynomials \emph{positive}. For example, the degree 1 positive polynomials are spanned over $\bb Q_{\geq 0}$ by $c_1$, while the degree 2 ones are spanned by $c_2$ and by $c_1^2-c_2$. In{\tiny } particular, $c_1^2=c_2+(c_1^2-c_2)$ is positive, but $c_1^2-2c_2$ is not.

Let $X$ be a smooth projective surface and let $\cal E$ be a nef vector bundle of rank $r$ on it. Then $c_1(\cal E)$ is nef, and $\int_X c_2(\cal E)$ and $\int_X(c_1^2(\cal E)-c_2(\cal E))$ are nonnegative integers. 	
If $\cal E$ is strongly slope semistable with respect to some polarization,  Bogomolov's  inequality gives $\int_X(2rc_2(\cal E)-(r-1)c_1^2(\cal E))\geq 0$.
However, there exist examples of strongly slope semistable vector bundles $\cal E$ on surfaces such that all the above positivity conditions hold for the characteristic classes of $\cal E$ without $\cal E$ being nef.

\begin{ex}\label{ex:notnef}
	Let $\cal E=\cal O_{\bb P^2}(-1)\otimes\bigotimes^3(T_{\bb P^2}(-1))$. This is strongly slope semistable since $T_{\bb P^2}$ is strongly slope semistable. Let $h=c_1(\cal O_{\bb P^2}(1))$.	Then $r={\rm rk}\,\cal E=8$, $c_1(\cal E)=4h$, and $c_2(\cal E)=16h^2$.
	We compute $c_1^2(\cal E)-c_2(\cal E)=0$. 
	So the characteristic classes of $\cal E$ suggest that $\cal E$ might be nef. However, the restriction of $\cal E$ to every line in $\bb P^2$ is $\cal O_{\bb P^1}(-1)\otimes\bigotimes^3(\cal O_{\bb P^1}(1)\oplus\cal O_{\bb P^1})$ and this has $\cal O_{\bb P^1}(-1)$ as a summand, hence it is not nef. \qed
\end{ex}

See also \cite[Section 5]{BHP14} for a related example.

\begin{conv}For the rest of this section we assume that $X$ is a smooth projective variety of dimension $n$.
\end{conv}

\subsection{Positive cones of cycles} Denote by $N^c(X)$ the space of numerical classes of cycles of codimension $c$ with real coefficients. 
For example $N^1(X)$ is the real N\' eron--Severi space spanned by Cartier divisors modulo numerical equivalence. The space $N^c(X)$ is finite dimensional. It contains important convex cones. For instance it contains the \emph{pseudoeffective cone} $\Eff^c(X)$, the closure of the cone spanned by classes of closed subsets of codimension $c$. It also contains the \emph{nef cone} $\Nef^c(X)$, the set of classes that intersect every $c$-dimensional subvariety nonnegatively, the dual of the pseudoeffective cone in the complementary codimension. The nef cone and the pseudoeffective cone in codimension $c$ each span $N^c(X)$ and do not contain linear subspaces of $N^c(X)$. See \cite{flpos}.

When $c=1$, we put $\Eff(X)=\Eff^1(X)$ and $\Nef(X)=\Nef^1(X)$.  In this case, $\Nef(X)\subseteq\Eff(X)$. The interior of the nef cone is the ample cone, and the interior of the pseudoeffective cone is the big cone of numerical classes with positive volume \cite[Chapter 2.2]{laz04}.

\subsection{Determinant and discriminant}
Let $\cal F$ be a torsion free sheaf on the smooth projective $X$.
Then the \emph{singular} (non-locally free) locus of $\cal F$ has codimension $\geq 2$. If $\cal F$ is reflexive, then the singular locus has codimension $\geq 3$. 

One can abstractly compute Chern classes of coherent sheaves in $A^*(X)$ by taking locally free resolutions and using the additivity of the Chern character ${\rm ch}$. The \emph{determinant} line bundle $\det\cal F$ of a torsion free sheaf can be defined by extending from the locally free locus, or considering the \emph{reflexive hull} $(\bigwedge^{\rm rk\,\cal F}\cal F)^{\vee\vee}$. This gives a concrete definition of $c_1$. The second Chern class of a reflexive sheaf can be similarly concretely defined.

\begin{rmk}\label{rmk:chernpullback}
	Let $f:Y\to X$ be a morphism of nonsingular projective varieties and let $\cal E$ be a coherent sheaf on $X$. Then $\ch({\bf L}f^*\cal E)=f^*\ch(\cal E)$ in the Chow ring of $Y$, where ${\bf L}f^*\cal E$ is the derived pullback. In particular $c_i(f^*\cal E)=f^*c_i(\cal E)$ for all $i$ in each of the following cases:
	\begin{enumerate}
		\item $\cal E$ is locally free.
		\item $f$ is flat.
		\item $f$ is the inclusion of a Cartier divisor and $\cal E$ is torsion free.
	\end{enumerate}
\end{rmk}

If $\imath:D\hookrightarrow X$ is the inclusion of a Cartier divisor and $\cal E$ is reflexive, then $\imath^*\cal E$ is torsion free. If $|H|$ is a basepoint free linear system and $\cal E$ is torsion free (resp.~reflexive), then $\imath^*\cal E$ is again torsion free (resp.~reflexive) for $D$ \emph{general} in $|H|$. See \cite[Corollary 1.1.14]{HL10}.

\medskip
 
The \emph{discriminant} of $\cal E$ is 
\[\Delta(\cal E)=2r c_2(\cal E)-(r-1)c_1^2(\cal E),\] where $r={\rm rk}\,\cal E$. It is an element of $A^2(X)$, but we use the same notation for its image in $N^2(X)$. The class 
\[\log r+\frac{c_1(\cal E)}r-\frac{\Delta(\cal E)}{2r^2}+\ldots\] 
in $A^*(X)_{\mathbb R}$
is the formal logarithm of the Chern character ${\rm ch}(\cal E)$. It follows that $\frac{\Delta(\cal E)}{2{\rm rk}^2\,\cal E}$ is additive for tensor products if one of the factors is locally free, just like the slope. In particular, $\Delta(\cal E)=\Delta(\cal E\otimes\cal O_X(D))$ for every divisor $D$ on $X$. Furthermore, note that if $c_1(\cal E)\equiv 0$, then $\Delta(\cal E)\equiv 0$ if and only if $c_2(\cal E)\equiv 0$. The vanishing $\Delta(\cal E)\equiv 0$ is equivalent to the numerical vanishing of the second Chern class $c_2$ of $\cal E\langle-\frac 1r\det\cal E\rangle$, the \emph{formal twist} of $\cal E$ in the sense of \cite[Chapter 6.2]{laz042}.
If $\cal E$ is locally free, then $\Delta(\cal E)\equiv 0$ is also equivalent to the numerical vanishing of $c_2$ of  $\cal E^{\otimes r}\otimes\det\cal E^{\vee}$ or  ${\cal E}{\rm nd}\, \cal E$.

\subsection{Semistability}

Let $H$ be an ample (or just nef) divisor on $X$. For a nonzero torsion free sheaf $\cal E$ on $X$, we define the \emph{slope} by $\mu_H(\cal E)=\frac{c_1(\cal E)\cdot H^{n-1}}{{\rm rk}\,\cal E}$. We say that $\cal E$ is $\mu_H$-\emph{semistable} (or \emph{slope semistable} with respect to $H$) if no proper subsheaf $0\neq\cal F\subsetneq\cal E$ verifies $\mu_H(\cal F)>\mu_H(\cal E)$. 
We say that $\cal E$ is $\mu_H$-\emph{stable} (or \emph{slope stable} with respect to $H$) if no proper subsheaf $0\neq\cal F\subsetneq\cal E$ with ${\rm rk}\,\cal F<{\rm rk}\,\cal E$ has $\mu_H(\cal F)\geq\mu_H(\cal E)$.
When $X$ is a curve and $H$ is ample, slope semistability is independent of $H$. In this case we simply say that $\cal E$ is semistable.
In positive characteristic it is useful to also consider Frobenius pullbacks.

\begin{defn}\label{strongly-ss-def}
	A torsion free sheaf $\cal E$ is called \emph{strongly $\mu_H$-(semi)stable} if $(F_X^m)^*\cal E$ is $\mu_H$-(semi)stable for all $m\geq 0$. Here $F_X$ denotes the Frobenius morphism in positive characteristic, and the identity morphism in characteristic zero.
\end{defn}

\noindent Examples show that these notions are indeed stronger that $\mu_H$-(semi)stability in positive characteristic. The first example showing this is due to  J.-P. Serre and it was published in \cite{Gieseker}. More precisely, the example shows that there exists a stable bundle of rank $2$ and degree $1$ on a genus $3$ curve in characteristic $3$, whose Frobenius pullback splits as direct sum of line bundles of different degrees. Nowadays there are many more examples of stable bundles that are not strongly semistable in an arbitrary positive characteristic. See, e.g., \cite[Theorem 1.1.3]{Joshi-Pauly} for recent examples that appear on any smooth projective curve in large characteristic.

\subsection{Semistability on curves}

In the sequel we will use the following well--known lemma. It goes back to R. Hartshorne \cite{Har70} in the characteristic zero case, with a subsequent algebraic proof in any characteristic due to C. M. Barton \cite[Theorem 2.1]{Barton71}. See also \cite[Proposition 7.1]{Moriwaki98}.

\begin{lem}\label{Hartshorne-Burton}
	Let $C$ be a smooth projective curve defined over some algebraically closed field $k$
and let $\cal  E$ be a degree $0$ vector bundle on $C$. Then
$\cal E$ is strongly semistable if and only if  $\cal E$ is nef.
\end{lem}

We will also use the following standard lemma (see, e.g., \cite[Lemma 7.1.2]{LePotier97}).

\begin{lem}\label{bound-on-sections}
	Let $\cal  E$ be a semistable vector bundle on a smooth projective curve $C$ defined over some algebraically closed field $k$. Then
	$$h^0(C, \cal E)\le {\rm rk}\,\cal E + \deg \cal E.$$
\end{lem}

\section{Numerically flat vector bundles on proper schemes}

Let $X$ be a proper scheme over a field $k$. 

\begin{defn}
	A vector bundle $\cal E$ on $X$ is called \emph{numerically flat} if both $\cal E$ and $\cal E^{\vee}$ are nef.
\end{defn}

\begin{rmk} A bundle $\cal E$ is numerically flat if and only if it is nef with $c_1(\cal E)\equiv 0$. Indeed if $\cal E$ and $\cal E^{\vee}$ are nef, then $c_1(\det\cal E)=c_1(\cal E)$ is numerically trivial. Conversely, if $\cal E$ is nef with $c_1(\cal E)\equiv 0$ then
 by \cite[Theorem 6.2.12]{laz042} any exterior power of $\cal E$ is nef	and hence
$\cal E^{\vee}\simeq\det\cal E^{\vee}\otimes\bigwedge^{r-1}\cal E$ is also nef. 
	
	It follows that if $\cal E$ is numerically flat, then all tensor functors of $\cal E$ and their duals are numerically flat, e.g., tensor powers, symmetric powers, divided powers, exterior powers, and all other Schur and co-Schur (Weyl) functors associated to $\cal E$.
	\end{rmk}

\subsection{Numerical triviality of Chern classes of numerically flat bundles}

\begin{prop}\label{prop:numflat}
	Let $X$ be a  proper scheme over a field $k$ and let $\cal E$ be a numerically flat vector bundle on $X$. Then $c_j(\cal E)$ is numerically trivial for all $j>0$.
\end{prop}

\begin{proof}
	The proof is by induction on $j$. For $j=1$ the assertion is clear since by assumption for any proper curve  $Y\subseteq X$ we have $\int _X c_1 (\cal E)\cap [Y]\ge 0$ and  
	$\int _X c_1 (\cal E^{\vee})\cap [Y]=-\int _X c_1 (\cal E)\cap [Y]\ge 0$.
	
	Let $Y\subseteq X$ be a $j$-dimensional subvariety of $X$. 	
	Let $s_j(\cal E^{\vee})=\pi_*(\xi^{j+r-1})$ be the $j$-th Segre class of the dual of $\cal E$.  
	By \cite[Chapter 3.2]{fulton84}, we have $s_j(\cal E^{\vee})=(-1)^{j+1}c_j(\cal E)-\sum_{i=1}^{j-1}(-1)^is_i(\cal E^{\vee})c_{j-i}(\cal E)$. By our induction hypothesis, we deduce $\int _{Y} s_j(\cal E_{Y}^{\vee})=(-1)^{j+1}\cdot\int _{Y} c_j(\cal E_{Y})$, so is suffices to prove that $\int _{Y}s_j(\cal E^{\vee}_{Y})=0$. 
	Set $Z\coloneqq{\bb P(\cal E_Y)}$ and $\cal L\coloneqq\cal O_{\bb P(\cal E_Y)}(1)$.
	Since $\cal E$ is nef, $\cal E_Y$ and $\cal L$ are also nef. 
	By the asymptotic Riemann--Roch theorems (see \cite[Chapter VI, Corollary 2.14 and Theorem 2.15]{kollarrational}) we have
	$$h^0(Z, \cal L ^{\otimes m})=\chi (Z, \cal L ^{\otimes m})+O (m^{j+r-2})= \frac{\int _{Z} c_1(\cal L) ^{j+r-1}}{(j+r-1)!} m^{j+r-1}+O (m^{j+r-2}).$$
	Since $\cal E_Y$ is numerically flat, ${\rm Sym} ^m\cal E_Y$ is also numerically flat. Hence for any ample divisor $H$ on $Y$ we have
	$ h^0 (Y, ({\rm Sym} ^m\cal E_Y)(-H))=0$ (otherwise ${\rm Sym} ^m\cal E_Y$ contains $\cal O(H)$ contradicting the nefness of $(\Sym^m\cal E_Y)^{\vee}$).
	Again using the asymptotic Riemann--Roch theorem we get 
	\begin{align*}
		h^0(Z, \cal L ^{\otimes m})=h^0 (Y, {\rm Sym} ^m\cal E_Y)&\le h^0 (Y, ({\rm Sym} ^m\cal E_Y)(-H))+ h^0 (H, {\rm Sym} ^m\cal E_H)\\
		&= h^0(\bb P(\cal E_H),
		\cal O_{\bb P(\cal E_H)}(m))= O (m^{j+r-2}).
	\end{align*}
	Summing up, we have  $\int _{Y} s_j(\cal E^{\vee}_{Y})=\int _{Z} c_1(\cal L) ^{j+r-1}=0$ as required.
\end{proof}

\begin{rmk}
	The idea of proof of vanishing of highest Segre classes comes from the proof of \cite[Proposition 5.1]{FulgerConesvect}, in turn inspired by the proof of the Bogomolov inequality. We avoid using this result and 
	semistability and give a proof working in an arbitrary characteristic. The proof of an analogue of \cite[Proposition 5.1]{FulgerConesvect} would require small rewriting and the use of deep results of Ramanan and Ramanathan \cite{RR84} on the behaviour of strong slope semistability.
\end{rmk}

\begin{rmk}
	If $X$ is a smooth complex manifold the above proposition was proven in \cite[Corollary 1.19]{DPS94} using earlier deep analytic results. If $X$ is a smooth variety and $k$ has positive characteristic the above proposition follows from \cite[Proposition 5.1 and Theorem 4.1]{Langerflat}. The proof of  \cite[Proposition 5.1]{Langerflat} cites rather deep results from \cite{fl83} although it uses only a much weaker and easier result of Kleiman \cite{Kleiman69}. However, it also depends on \cite{Langer04} and the proof above is much more elementary.
\end{rmk}

\begin{rmk}
	From \cite{BlochGieseker} or \cite{fl83}, the Chern class $c_k(\cal E)$ is nef if $\cal E$ is a nef bundle. Thus for $\cal E$ numerically flat the classes $c_k(\cal E)$, $c_k(\cal E^{\vee})$, and $s_k(\cal E^{\vee})$ are nef.
	Using induction and the fact that $\Nef^k(X)$ does not contain linear subspaces (cf.~\cite{flpos}), this gives another argument than the one above.
	Note though that \cite{BlochGieseker} and consequently also \cite{fl83} use the hard Lefschetz theorem on cohomology so this proof is much harder. 
\end{rmk}

\medskip

Together with the main result of \cite{Langer04} the above propostion implies the following corollary:

\begin{cor}\label{boundedness}
	Let $f:X\to S$ be a flat projective morphism of noetherian schemes. Then the set of numerically flat vector bundles of fixed rank $r$ on the fibers of $f$ is bounded.
\end{cor}

\begin{proof}
	Let $\cal O_X (1)$ be an $f$-very ample line bundle on $X$ and let $\cal E$ be  a rank $r$  numerically flat vector bundle on a geometric fiber $X_s$ for some geometric point  $s$ of $S$. The singular Grothendieck--Riemann--Roch theorem (see \cite[Corollary 18.3.1]{fulton84}) and Proposition \ref{prop:numflat} imply that 
	$$\chi (X_s, \cal E (m))=\int _{X_s} \ch (\cal E(m))\cap {\mathrm {Td}} ({X_s})=r \int _{X_s} \ch ( \cal O_{X_s} (m))\cap {\mathrm {Td}} ({X_s})= r \chi ({X_s}, \cal O_{X_s} (m)).$$
Since $f$ is flat, for every connected component $S_0$ of $S$ the Hilbert polynomial  $P_{s}(m)=\chi ({X_s}, \cal O_{X_s} (m))$ is independent of the geometric point  $s$ of $S_0$.  Moreover, any numerically flat vector bundle on $X_s$ is slope $\cal O_{X_s} (1)$-semistable (the general definition of slope semistability in case of singular projective schemes can be found in \cite[Definition and Corollary 1.6.9]{HL10}). Therefore the required assertion follows from \cite[Theorem 4.4]{Langer04}.
\end{proof}

\begin{rmk}
	In case $X$ is a normal variety and $S=\Spec k$, the above corollary was proved in \cite[Theorem 1.1]{Langerflat2}. If $S=\Spec k$ and $k$ is a finite field the above corollary was proved in \cite[Theorem 2.4]{De-We}. 
\end{rmk}

\begin{thrm}
	Let  $X$ be a projective scheme over a perfect field $k$ of positive characteristic. 	Let $\cal E$ be a rank $r$  vector bundle on $X$.
	Then the following  conditions are equivalent:
	\begin{enumerate}
		\item $\cal E$ is numerically flat.
		\item The set $\{(F_X^m)^*\cal E\}_{m\in \mathbb Z_{\ge 0}}$ is bounded.
		\item There exist $m_1>m_2\ge 0$ such that $(F_X^{m_1})^*\cal E $ and $(F_X^{m_2})^*\cal E $ are algebraically equivalent.
	\end{enumerate}
\end{thrm}

\begin{proof}
	If $\cal E$ is numerically flat then all $\cal E_m=(F_X^m)^*\cal E$ are numerically flat, so 
	$(1)\Rightarrow (2)$ follows from Corollary \ref{boundedness}. 
	Assume $(2)$. Then by definition there exists a $k$-scheme $S$ of finite type and an $S$-flat coherent sheaf $\cal F$ on $X_S:=X\times _kS$ such that for every $m\in \mathbb Z$
	there exists a geometric  $k$-point $s_m$ in $S$ such that $\cal F_{X_{s_m}}\simeq \cal E_m$. Now $(3)$ follows by the pigeonhole principle applied to the finitely many connected components of $S$.
	Assume $(3)$. Then for all $m\ge 0$ the bundles $(F_X^{m_1+m})^*\cal E $ and $(F_X^{m_2+m})^*\cal E $ are algebraically equivalent. This implies that the family $\{ (F_X^{m_2+m(m_1-m_2)})^*\cal E  \}_{m\in \mathbb Z _{\ge 0}}$ is bounded. The implication $(3)\Rightarrow (1)$ follows as in the first part of proof of \cite[Proposition 5.1]{Langerflat}. 
\end{proof}

The following corollary  can be thought of as a generalization of \cite[Theorem 2.3]{De-We}
from finite fields to arbitrary perfect fields of positive characteristic. In case $X$ is smooth the result is contained in \cite[Corollary 3.7]{La-Chern}. 

\begin{cor}\label{cor-Bloch}
Let  $X$ be a projective scheme over a perfect field $k$ of positive characteristic. 
	Let $\cal E$ be a numerically flat vector bundle on $X$.
	Then for all $i>0$ the Chern classes $c_i (\cal E)$ are, up to torsion, algebraically equivalent to $0$.
\end{cor}

\begin{proof}
	By the above theorem we know that  there exist $m_1>m_2\ge 0$ such that $(F_X^{m_1})^*\cal E $ and $(F_X^{m_2})^*\cal E $ are algebraically equivalent.
	Since $c_i ((F_X^{m})^*\cal E)=p^{im}c_i(\cal E)$ in $A^*(X)$, we get
	$$0=c_i ((F_X^{m_1})^*\cal E)-c _i ((F_X^{m_2})^*\cal E)= (p^{im_1}-p^{im_2})c_i(\cal E)$$
	in $B^*(X)$, so $c_i(\cal E)=0 $ in $B^*(X)_{\mathbb Q}$.
\end{proof}

\subsection{Characterizations of numerically flat bundles}

Let $X$ be a  proper scheme over an algebraically closed field $k$.

\begin{defn}
A vector bundle $\cal E$ on $X$ is called \emph{universally semistable} if for all $k$-morphisms $f: C\to X$ from smooth connected projective curves $C$ over $k$ the pullback $f^*\cal E$ is semistable. We say that $\cal E$
is \emph{Nori semistable} if it is universally semistable and $c_1(\cal E)$ is numerically trivial.
\end{defn}

To better justify the terminology, note that if $\cal E$ is universally semistable, and $f:Y\to X$ is a morphism from a projective manifold $Y$ over $k$, then for every polarization $H$ on $Y$ the pullback $f^*\cal E$ is $\mu_H$-semistable.

If  $f: Y\to X$ is a proper generically finite morphism between smooth projective varieties and $\cal E$ is a strongly slope $H$-semistable bundle on $X$ then 
$f^*\cal E$ is slope $f^*H$-semistable bundle. This motivates the following definition:

\begin{defn}
If $X$ is irreducible then we say that a vector bundle $\cal E$ on $X$ is \emph{strongly semistable} if there exist  a proper generically finite $k$-morphism $f: Y\to X$ from a smooth projective $k$-variety $Y$ to $X$ and an ample divisor $H$ on $Y$ such that the bundle $f^*\cal E $ is strongly slope $H$-semistable.
In general, we say that a vector bundle  $\cal E$ on $X$ is \emph{strongly semistable} if
its restriction to every irreducible component of $X$ is strongly semistable.
\end{defn}

A line bundle $\cal L$ is said to be $\tau$-trivial, if $\cal L^{\otimes m}$ is algebraically equivalent to $\cal O_X$ for some $m\geq 1$. This notion is equivalent to $\cal L$ being numerically trivial by \cite[Theorem 6.3]{Kleiman-Picard}.
Numerically flat bundles can be seen as a higher rank version of $\tau$-trivial bundles.
In the proof we use Theorem \ref{thm:main3}.

\begin{thrm}\label{main:proper-num-flat}
Let $X$ be a proper scheme over an algebraically closed field $k$. Let $\cal E$ be a rank $r$  vector bundle on $X$. Then the following conditions are equivalent:
\begin{enumerate}
\item $\cal E$ is numerically flat.
\item  $\cal E$ is Nori semistable. 
\item  $\cal E$ is strongly semistable and $c_j(\cal E)$ is numerically trivial for all $j>0$.
\item  $\cal E$ is strongly semistable and both $c_1(\cal E)$ and $c_2(\cal E )$ are numerically trivial.
\item $\cal E$ is strongly semistable and for every coherent sheaf $\cal F$ on $X$ we have
$\chi (X, \cal E\otimes \cal F)=r\cdot\chi (X, \cal F)$.
\end{enumerate}
\end{thrm}

\begin{proof}
	The equivalence of $(1)$ and $(2)$ is well known and it follows from Lemma \ref{Hartshorne-Burton} (see, e.g., \cite[1.2]{Langerflat}).
Assume that $\cal E$ is numerically flat and let $X_0$ be some irreducible component of $X$. By \cite{dejong96} there exists  a proper generically finite $k$-morphism $f: Y\to X_0$ from a smooth projective $k$-variety $Y$. Then  $f^*\cal E$ is numerically flat, so 
it is strongly $\mu_H$-semistable for every ample divisor $H$ on $X$. In particular, $\cal E$ is strongly semistable. Now implication $(1)\Rightarrow (3)$ follows from Proposition \ref{prop:numflat}.

In proof of the implication  $(3)\Rightarrow (5)$ we use singular Riemann--Roch \cite[Chapter 18]{fulton84}.
Let $f: X\to Y=\Spec k$ be the structural morphism. We denote by $[\cal E]$ the class of $\cal E$ in $K^0(X)$ and by $[\cal F]$ the class of $\cal F$ in $K_0(X)$. 
By \cite[Theorem 18.3]{fulton84} we  have canonical maps $\tau_X: K_0(X)\to A_*(X)_{\mathbb Q}$ and $\tau_Y: K_0(Y)\to A_*(Y)_{\mathbb Q}=\mathbb Q$, which satisfy the following equalities:
$$\chi (X, \cal E\otimes \cal F)=\tau _Yf_*([\cal E]\otimes [\cal F])= f_*\tau _X([\cal E]\otimes [\cal F])= f_*(\ch ([\cal E])\cap \tau _X([\cal F]))=r \tau _Yf_*([\cal F])=
r\chi (X, \cal F).$$

To prove that $(5)$ implies $(1)$ we can assume that $X$ is irreducible.
Then by assumption  there exist a proper generically finite $k$-morphism $f: Y\to X$ from a smooth projective $k$-scheme $Y$ to $X$ and an ample divisor $H$ on $Y$ such that the bundle $f^*\cal E $ is strongly slope $H$-semistable. By the Leray spectral sequence and the projection formula we have
\begin{align*}
	\chi (Y, f^*\cal E(mH))&=\sum _i(-1)^i \chi (X, R^if_*(f^*\cal E(mH)))=
	\sum _i(-1)^i \chi (X,\cal E\otimes  R^if_*\cal O_Y (mH))\\
	&=
	r\sum _i(-1)^i \chi (X,  R^if_*\cal O_Y (mH))= r \chi (Y, \cal O_Y (mH)).
\end{align*}
So by the implication $(4)\Rightarrow (2)$ of Theorem \ref{thm:main3} we see that 
$f^*\cal E$ is numerically flat. Since $f$ is surjective, this implies that $\cal E$
is also numerically flat.

The implication  $(3)\Rightarrow (4)$ is obvious, so it is sufficient to prove that  $(4)\Rightarrow (1)$. Without loss of generality we can assume that $X$ is irreducible and
there exist a proper generically finite $f: Y\to X$ from a smooth projective $k$-variety $Y$ and an ample divisor $H$ on $Y$ such that the bundle $f^*\cal E $ is strongly $\mu_H$-semistable. Then $c_1(f^*\cal E)$ and $c_2(f^*\cal E )$ are numerically trivial. So by the implication $(1)\Rightarrow (2)$ of Theorem \ref{thm:main3} we see that 
$f^*\cal E$ is numerically flat. As before this implies that $\cal E$
is also numerically flat.
\end{proof}

\begin{cor} \label{main:proper-universally-ss}
Let $\cal E$ be a vector bundle on $X$. Then the following conditions are equivalent:
	\begin{enumerate}
	\item $\cal E$ is strongly semistable and $\Delta(\cal E)$ is numerically trivial.
	\item $\cal E{\rm nd}\, \cal E$ is nef.
	\item  $\cal E$ is universally semistable. 
	\end{enumerate}
\end{cor}

\begin{proof}
	If  $\cal E$ is strongly semistable and $\Delta(\cal E)\equiv 0$ then 
	$\cal E{\rm nd}\, \cal E$ is strongly semistable and both $c_1(\cal E{\rm nd}\, \cal E)$ and $c_2(\cal E{\rm nd}\, \cal E )$ are numerically trivial. So $\cal E{\rm nd}\, \cal E$ is numerically flat, which proves $(1)\Rightarrow (2)$.
	If $\cal E{\rm nd}\, \cal E$ is nef then it is also numerically flat (as it isomorphic to its dual) and hence it is Nori semistable. This implies $(3)$.
	To prove that  $(3)\Rightarrow (1)$ it is sufficient to prove that $\Delta(\cal E)\equiv 0$. But $\cal E{\rm nd}\, \cal E$ is universally semistable with trivial determinant, so it is Nori semistable. Hence it is numerically flat and
	$\Delta(\cal E)=c_2(\cal E{\rm nd}\, \cal E) \equiv 0$.
\end{proof}

\section{Algebraic proofs of Theorems \ref{thm:main} and \ref{thm:main2}}

\subsection{Restriction theorems}

In the sequel we frequently use the following strengthening of the Mehta--Ramanathan theorem for sheaves with vanishing discriminant. The result follows from \cite[Theorem 5.2]{Langer04} with a different proof from the Mehta--Ramanathan theorem.

\begin{lem}\label{restriction}
Let $X$ be a smooth projective variety defined over an algebraically closed field $k$
and let $H$ be an ample divisor class on $X$. Let $\cal E$ be a torsion free sheaf with $\Delta(\cal E)\equiv 0$. Then there exists $m_0=m_0(X,H,{\rm rk}\,\cal E)\geq 1$ such that for all $m\geq m_0$
\begin{enumerate}
	\item If $\cal E$ is strongly $\mu_H$-stable and $\cal E_D$ is torsion free for some normal divisor $D\in |mH|$,
	then $\cal E_D$ is strongly $\mu _{H_D}$-stable.
	\item If $\cal E$ is strongly $\mu_H$-semistable, 
	then for general $D\in |mH|$  the restriction $\cal E_D$ is strongly $\mu_{H_D}$-semistable.
\end{enumerate}  
\end{lem}

\begin{proof}In characteristic zero pick $m_0$ such that $|mH|$ is basepoint free for all $m\geq m_0$. In positive characteristic we also need to exceed a constant $\beta_r$ depending on $X$, $H$, and ${\rm rk}\,\cal E$. See the inequality in \cite[Theorem 5.2]{Langer04}.
(1) follows immediately from \cite[Theorem 5.2]{Langer04}. 
We obtain (2) as a consequence of (1) as in \cite[Corollary 5.4]{Langer04}. 
The factors (successive quotients) in any Jordan--H\"older filtration of $\cal E$ are $\mu_H$-stable. As in \cite[Theorem 5.4]{Langer04} we observe that they also have numerically trivial discriminant. 
Their restriction to a general $D$ in a basepointfree $|mH|$ is again torsion free.
In characteristic zero then (2) follows from (1).
In positive characteristic strong $\mu_H$-semistability also takes into account the countably many Frobenius pullbacks.
We remark that there exists some $s_0$ such that the factors in a Jordan--H\"older filtration of $(F_X^{s_0})^*\cal E$ are strongly $\mu_H$-stable. Then for a general divisor  $D\in |mH|$ the restrictions of the factors in a Jordan--H\"older filtration of the sheaves $\{(F_X^s)^*\cal E\} _{s\le s_0}$ to $D$  are torsion free. The restrictions of the factors in a Jordan--H\"older filtration of  $(F_X^s)^*\cal E$ for all $s\ge s_0$ to such a divisor are also torsion free since they are pullbacks of those of $(F_X^{s_0})^*\cal E$.   
\end{proof}

\subsection{The surface case}

\begin{prop}\label{prop:surfacetocurve}
	Let $(X,H)$ be an amply polarized smooth projective surface defined over an algebraically closed field $k$. Let $\cal E$ be a strongly $\mu_H$-semistable locally free sheaf of rank $r$ with $c_i(\cal E)\equiv 0$ for $i=1,2$. Then
	\begin{enumerate}
		\item For any line bundle $L$  on $X$ and any $i\in\{0,1,2\}$
		we have $h^i(X,L\otimes\Sym^m\cal E)=O(m^{r-1})$.
		\item $\cal E$ is numerically flat.
	\end{enumerate}
\end{prop}

\begin{proof}
Without loss of generality we can assume that  $H$ is very ample.
Since $\dim\bb P(\cal E)=r+1$, the claimed growth rate in $(1)$ is two degrees lower than expected. The proof is similar to the proof of the Bogomolov inequality in \cite[Theorem 7.3.1]{HL10}.

By \cite[Theorem 3.23 and the remark at the end of Section 4]{RR84} the bundle
$\Sym^m\cal E$ is strongly $\mu_H$-semistable. Since 	$\Delta (\Sym^m\cal E)\equiv 0$ Lemma \ref{restriction} implies that if $D\in |H|$ is a general divisor then $(\Sym^m\cal E)_D$ is strongly semistable of degree $0$.
By Lemma \ref{bound-on-sections} we have
$$ h^0(D,(L\otimes \Sym^m\cal E)_D)\le  (1+\deg L_D) \cdot {\rm rk}\, \Sym^m\cal E  =O(m^{r-1})$$
From the short exact sequence
\[0\to(L\otimes\Sym^m\cal E)(-H)\to L\otimes\Sym^m\cal E\to (L\otimes\Sym^m\cal E)_D\to 0\]
we have 
$$h^0(X,L\otimes\Sym^m\cal E)\leq h^0(D,(L\otimes\Sym^m\cal E)_D)+h^0(X,(L\otimes\Sym^m\cal E)(-H)).$$ 
Changing $H$ by its multiple (which does not depend on $m$) if necessary, we can assume that $L(-H)$ has negative degree with respect to $H$ (e.g., we can assume that $L^{\vee}(H)$ is effective).
Then the bundle $(L\otimes \Sym^m\cal E)(-H)$ is $\mu_H$-semistable with negative slope so it does not have any nonzero sections. Therefore $h^0(X,L\otimes \Sym^m\cal E)=O(m^{r-1})$.

By Serre's duality, $h^2(X,L\otimes\Sym^m\cal E)=h^0(X,(\Sym^m\cal E)^{\vee}\otimes\omega_X\otimes L^{\vee})$. The bundle $(\Sym^m\cal E)^{\vee}$ is also strongly semistable with numerically trivial Chern classes.  An analogous proof to the case $i=0$ gives  
$h^2(X,L\otimes \Sym^m\cal E)=O(m^{r-1})$. Note that in positive characteristic this equality does not follow formally from the previous case applied to $\cal E^{\vee}$.

Finally, to prove that $h^1(X,L\otimes\Sym^m\cal E)$ grows at most like $O(m^{r-1})$,
it is sufficient to prove that  $\chi(X,L\otimes\Sym^m\cal E)=O(m^{r-1})$.
	Let $\pi: \bb P(\cal E)\to X$ be the bundle map  and let us set $\xi\coloneqq c_1(\cal O_{\bb P(\cal E)}(1))$.
Since the Chern classes $c_1(\cal E)$ and $c_2(\cal E)$ are both numerically trivial, we have $\xi ^{r}=0$. By the Riemann--Roch theorem
\[\chi(X,L\otimes\Sym^m\cal E)=\chi(\bb P(\cal E),\cal O_{\bb P(\cal E)}(m)\otimes\pi^*L)=\int_{\bb P(\cal E)}\exp(m\xi+\pi^*c_1(L))\cap\td \bb P(\cal E),\]
so the coefficients of $m^{r+1}$ and $m^r$ in the expression above are $0$. The claim is proved.
\smallskip

$(2)$.  Let  $f:C\to X$ be a morphism from a smooth projective curve and let $C'\subset X$ be the (possibly singular) image of $f$.
Consider the restriction sequence $$0\to(\Sym^m\cal E)(-C')\to\Sym^m\cal E\to\Sym^m\cal E_{C'}\to 0.$$ Let $\xi_{C}=c_1(\cal O_{\bb P(f^*\cal E)}(1))$. If $f^*\cal E$ is not strongly semistable, then since it has degree 0, it follows from Lemma \ref{Hartshorne-Burton} that $\xi_C$ is not nef. It is however big since some Frobenius pullback of $f^*\cal E$ has a strongly semistable subbundle of positive degree, so an ample subbundle. Bigness is invariant under birational pullback (cf.~\cite[Chapter 2.2]{laz04}) and even under dominant generically finite pullback, hence $h^0(C',\Sym^m \cal E_{C'})=h^0(\bb P(\cal E_{C'}),\cal O_{\bb P(\cal E_{C'})}(m))$ grows like $O(m^r)$. 
	However, we have  $$h^0(C',\Sym^m\cal E_{C'})\leq h^0(X,\Sym^m\cal E)+h^1(X,(\Sym^m\cal E)(-C')).$$ We get a contradiction from part $(1)$.
\end{proof}

\subsection{Local freeness via the vanishing of the discriminant}

\begin{lem}\label{lem:pulltoproduct}
	Let $(X,H)$ and $(Y,A)$ be amply polarized smooth projective varieties defined over an algebraically closed field $k$. Let $\cal E$ be a $\mu_H$-semistable ($\mu_H$-stable, strongly $\mu_H$-semistable or strongly $\mu_H$-stable) torsion free  sheaf on $X$. 
	Then ${\rm pr}_X^*\cal E$ is $\mu_L$-semistable (respectively $\mu_L$-stable, strongly $\mu_L$-semistable or strongly $\mu_L$-stable) for $L={\rm pr}_X^*H+{\rm pr}_Y^*A$. 
\end{lem}

\begin{proof}
	Let $r={\rm rk}\,\cal E$ and denote $n=\dim X$ and $m=\dim Y$.
	Let $\cal F\subseteq{\rm pr}_X^*\cal E$ be a subsheaf of rank less than $r$. For $x\in X$ and $y\in Y$, let $\cal F_x\coloneqq\cal F_{\{x\}\times Y}$ and $\cal F_y\coloneqq\cal F_{X\times\{y\}}$. For general points $x\in X(k)$ and $y\in Y(k)$ we have $\cal F_x\subseteq\cal O_Y^{\oplus r}$ and $\cal F_y\subseteq\cal E$.
	Since $\cal E$ is $\mu_H$-semistable and $\cal O_Y^{\oplus r}$ is $\mu_A$-semistable, we deduce
	that $\frac{c_1(\cal F)\cdot {\rm pr}_X^*H^{n-1}{\rm pr}_Y^*A^m}{(A^m)\cdot {\rm rk}\,\cal F}=\frac{c_1(\cal F_y)\cdot H^{n-1}}{{\rm rk}\,\cal F}\leq \mu_H(\cal E)$ and $\frac{c_1(\cal F)\cdot {\rm pr}_X^*H^{n}{\rm pr}_Y^*A^{m-1}}{(H^n)\cdot {\rm rk}\,\cal F}=\frac{c_1(\cal F_x)\cdot A^{m-1}}{{\rm rk}\,\cal F}\leq0$.
	Then 
	\begin{multline*}
		\mu_L(\cal F)=\frac{c_1(\cal F)\cdot L^{n+m-1}}{{\rm rk}\,\cal F}=\frac{c_1(\cal F)\cdot \left({{n+m-1}\choose{n-1}}{\rm pr}_X^*H^{n-1}{\rm pr}_Y^*A^m+{{n+m-1}\choose n}{\rm pr}_X^*H^n{\rm pr}_Y^*A^{m-1}\right)}{{\rm rk}\,\cal F}\\ 
		\leq {{n+m-1}\choose{n-1}}(A^m)\cdot \mu_H(\cal E)=\frac{{{n+m-1}\choose{n-1}}\cdot (c_1(\cal E)\cdot H^{n-1})(A^m)}{{\rm rk}\,\cal E}=\mu_L({\rm pr}_X^*\cal E).
	\end{multline*}
If $\cal E$ is $\mu_H$-stable then  $\frac{c_1(\cal F)\cdot {\rm pr}_X^*H^{n-1}{\rm pr}_Y^*A^m}{(A^m)\cdot {\rm rk}\,\cal F}=\frac{c_1(\cal F_y)\cdot H^{n-1}}{{\rm rk}\,\cal F}< \mu_H(\cal E)$ and we get $	\mu_L(\cal F)<\mu_L({\rm pr}_X^*\cal E)$ as required.

Applying the above assertions for slope semistability and slope stability to all Frobenius pull-backs gives immediately the assertions for strong slope semistability and strong slope stability.
\end{proof}

\begin{prop}\label{thm:main-variation}
	Let $X$ be a smooth projective variety of dimension $n$ defined over an algebraically closed field $k$
	and let $H$ be an ample polarization on $X$.
	 Let $\cal E$ be a strongly $\mu_H$-stable torsion free sheaf on $X$ with $c_1 (\cal E)\equiv 0$. Then the following conditions are equivalent:
	\begin{enumerate}
		\item $\cal E$ is reflexive and $c_2(\cal E)\cdot H^{n-2}=0$.
	    \item $\cal E$ is locally free and numerically flat.
		\item  $c_j(\cal E)\equiv 0$ for all $j\ge 1$.
		\item The normalized Hilbert polynomial $\frac 1{{\rm rk}\,\cal E}\cdot\chi(X,\cal E(mH))$ of $\cal E$ is equal to the Hilbert polynomial of $\cal O_X$.
	\end{enumerate} 
\end{prop}

\begin{proof}
We argue by induction on $n$. If $n=1$, the equivalence of all four conditions is tautological, with $(2)$ being implied by Lemma \ref{Hartshorne-Burton}.
Let us assume that $n=2$. Since  every reflexive sheaf on a smooth surface is locally free, the implication 
$(1)\Rightarrow (2)$ follows from Proposition \ref{prop:surfacetocurve}.
The implication $(2)\Rightarrow (3)$ follows from Proposition \ref{prop:numflat} and $(3)\Rightarrow (4)$ follows from the Hirzebruch--Riemann--Roch theorem.
To prove $(4)\Rightarrow (1)$, consider the exact sequence
$0\to\cal E\to \cal E^{\vee\vee}\to Q\to 0,$
where $Q$ has a finite support.  Since $c_1(\cal E^{\vee\vee})=c_1(\cal E)\equiv 0$ and $\cal E^{\vee\vee}$ is strongly $\mu_H$-stable, the  Bogomolov type inequality (see \cite[Theorem 3.2]{Langer04}) gives $\int _X \Delta(\cal E^{\vee\vee})\geq 0$. Our assumption on the Hilbert polynomial of $\cal E$ and the Hirzebruch--Riemann--Roch theorem imply that $\int _X c_2(E)=0$. Therefore with $r={\rm rk}\,\cal E$
	$$\int _X \Delta(\cal E^{\vee\vee})=2r\int _Xc_2(\cal E^{\vee\vee})=2r\int _X(c_2(\cal E)+c_2(Q))=2r\int _Xc_2(Q)\leq 0,$$ 
which gives $\int _Xc_2(Q)=0$. But then $Q=0$ and $\cal E$ is reflexive.
 
 \medskip
 
Let us now assume that the result holds for varieties of dimension $<n$, where $n\ge 3$. 

$(1)\Rightarrow (2)$

Let $\cal E$ be reflexive, strongly  $\mu_H$-stable, with $c_1(\cal E)$ numerically trivial and $c_2(\cal E)\cdot H^{n-2}=0$.
Let $\imath:D\to X$ be the inclusion of a general member of $|mH|$ for sufficiently large $m$. Then $D$ is smooth of dimension $n-1$. The restriction $\imath^*\cal E$ is still reflexive by the general choice of $D$ and  it is strongly $\mu_{\imath^*H}$-stable by Lemma \ref{restriction}. We also have  $c_2(\imath^*\cal E)\cdot\imath^*H^{n-3}=c_2(\cal E)\cdot H^{n-2}$ by Remark \ref{rmk:chernpullback}. Using the implication $(1)\Rightarrow (4)$ on $D$ we see that 
\begin{align*}
\chi(X, \cal E(mH))- \chi (\cal E ((m-1)H))&= \chi (D, \imath^*\cal E(mH))= r \chi (D, \cal O_D(mH))\\
&=r(\chi(X, \cal O_X(mH))- \chi (\cal O_X ((m-1)H))).
\end{align*}
Now assume that $\cal E$ is not locally free. For sufficiently large $m$, let
$\imath:D\to X$ be an embedding of a member of $|mH|$ that is general among those that pass through one of the points where $\cal E$ is not locally free. By \cite{Bertini} we know that $D$ is smooth. We also know that the restriction $\imath^*\cal E$ is torsion-free.
So  Lemma \ref{restriction} implies that  $\imath^*\cal E$ is strongly $\mu_{\imath^*H}$-stable. By the above, we also know that 
\begin{align*}
 \chi (D, \imath^*\cal E(mH))&= \chi(X, \cal E(mH))- \chi (\cal E ((m-1)H))\\
&=r(\chi(X, \cal O_X(mH))- \chi (\cal O_X ((m-1)H)))=r \chi (D, \cal O_D(mH)).
\end{align*}
Then the induction assumption implies that 
$\imath^*\cal E$ is locally free. By \cite[Lemma 1.14]{Langer19} we deduce that $\cal E$ is locally free around $D$, a contradiction.
Thus $\cal E$ is locally free and we need to prove that it is numerically flat.
Let $f: C\to X$ be a morphism from a smooth projective curve.
Let $A$ be any ample polarization on $C$ and let  $\widetilde C\subset X\times C$ be the embedding of the graph of $C$. Denote by ${\rm pr}_X:X\times C\to X$ and 
${\rm pr}_C:X\times C\to C$ the projections onto the two factors. Let us also set $L\coloneqq{\rm pr}_X^*H+{\rm pr}_C^*A$ and $\widetilde {\cal E}\coloneqq  {\rm pr}_X^*\cal E$. 
Then $\widetilde C\simeq C$ and $f^*\cal E\simeq \widetilde {\cal E}_{\widetilde C}$,
so it is sufficient to check that $\widetilde {\cal E}_{\widetilde C}$ is semistable of degree $0$.   We have $c_2(\widetilde  {\cal E})\cdot L^{n-1}=\sum_{j=0}^{n-1}{{n-1}\choose j}{\rm pr}_X^*(c_2(\widetilde  {\cal E})\cdot H^{n-1-j})\cdot{\rm pr}_C^*A^j$. For dimension reasons, all terms except possibly $j=1$ vanish. The term $j=1$ also vanishes from the assumption $c_2(\cal E)\cdot H^{n-2}=0$. Therefore we have  $c_1(\widetilde {\cal E})\equiv 0$ and $c_2(\widetilde  {\cal E})\cdot L^{n-1}=0$. By Lemma \ref{lem:pulltoproduct} we also know 
that $\widetilde {\cal E}$ is strongly $\mu_L$-stable. 
By  \cite[Theorem 3.1]{Bertini} for large $m$ there exists a chain of smooth 
varieties $S=X_2\subset X_3\subset ...\subset X_n=X\times C$ containing $\widetilde C$ such that all $X_i\in |mL_{X_{i+1}}|$ are smooth. Then  Lemma \ref{restriction} implies that $\widetilde  {\cal E}_S$ is strongly $\mu _{L_S}$-stable with  $c_1 (\widetilde  {\cal E}_S)\equiv 0$
and  $\int_S c_2 (\widetilde  {\cal E}_S)=0$. So the required assertion follows from Proposition  \ref{prop:surfacetocurve}.
\smallskip

$(2)\Rightarrow (3)$ follows from Proposition \ref{prop:numflat}.
\smallskip	
	
$(3)\Rightarrow (4)$ follows from the Hirzebruch--Riemann--Roch theorem.
\smallskip
 
$(4)\Rightarrow (1)$ 
Let $\cal E$ be a rank $r$ torsion free, strongly $\mu_H$-stable sheaf  with $c_1(\cal E)\equiv 0$ and
$\chi (X, \cal E (mH))= r \chi (X, \cal O_X (mH))$ for all $m\in \mathbb Z$. Comparing coefficients of these polynomials at $m^{n-2}$ we see that  $c_2(\cal E)\cdot H^{n-2}=0$. Consider the exact sequence
$$0\to\cal E\to\cal E^{\vee\vee}\to Q\to 0,$$
where $Q$ is a torsion sheaf on $X$.  Then $\cal E^{\vee\vee}$ is reflexive, strongly $\mu_H$-stable sheaf  with $c_1(\cal E^{\vee\vee})\equiv 0$. 
As in the surface case, the Bogomolov type inequality (see \cite[Theorem 3.2]{Langer04}) gives
$$ \Delta(\cal E^{\vee\vee})\cdot H^{n-2}=2r c_2(Q)\cdot H^{n-2}\ge 0,$$ 
which implies $\Delta(\cal E^{\vee\vee})\cdot H^{n-2}=c_2(Q)\cdot H^{n-2}=0$.
Using already proven implication $(1)\Rightarrow (4)$ we see that  Hilbert polynomials of $\cal E$
and $\cal E ^{\vee\vee}$ coincide. Therefore the Hilbert polynomial of $Q$ vanishes. This implies that $Q=0$ and hence  $\cal E$ is reflexive with $c_2(\cal E)\cdot H^{n-2}=0$.
\end{proof}

The following known example shows that reflexivity assumption is necessary in condition $(1)$ of Proposition \ref{thm:main-variation}.

\begin{ex}
Let $X$ be a smooth projective variety of dimension $n\geq 3$. Let $\cal I\subset\cal O_X$ be the ideal sheaf of a nonempty closed subset of codimension $j\geq 3$. Then $\cal I$ is torsion-free, strongly slope semistable with respect to any polarization, and $c_1(\cal I)$ and $c_2(\cal I)$ are numerically trivial. However $c_j(\cal I)$ is not numerically trivial and of course $\cal I$ is not locally free.
\end{ex}

\begin{cor}\label{cor:sstablec1=0}
 Let $\cal E$ be a reflexive strongly $\mu_H$-semistable sheaf  with $c_1(\cal E)\equiv 0$ and $c_2(\cal E)\cdot H^{n-2}=0$. Then $\cal E$ is locally free  and numerically flat.
 Moreover, every factor in a Jordan--H\"older filtration of $\cal E$ is also locally free  and numerically flat.
\end{cor}
This proof is analogous to \cite[Proposition 2.5]{NakayamaNormalized} and
\cite[Section 2.2]{Langer19}.
\begin{proof}
	The statement is clear if $n=1$ so we assume that $n\geq 2$.
	We perform induction on the rank $r$ of $\cal E$. The case $r=1$ or more generally $\cal E$ is strongly $\mu_H$-stable is Proposition \ref{thm:main-variation}.
	Thus we may assume that $\cal E$ is not strongly $\mu_H$-stable. 
	First, let us assume that there exists an exact sequence
	\[0\to\cal E_1\to\cal E\to \cal E_2\to 0\]
	where $\cal E_1$ is $\mu_H$-stable, strongly $\mu_H$-semistable and reflexive and $\cal E_2$ is strongly $\mu_H$-semistable and torsion free, both $\cal E_1$ and $ \cal E_2$ are nonzero sheaves and $\mu_H( \cal E_1)=\mu_H( \cal E_2)=0$.
	By the Hodge index theorem we have
	\begin{align*}
	0&=\frac{\Delta (\cal E)\cdot H^{n-2}}{r}=\frac{\Delta (\cal E_1)\cdot H^{n-2}}{r_1}+\frac{\Delta (\cal E_2)\cdot H^{n-2}}{r_2}-\frac{r_1r_2}{r}\left( \frac{c_1(\cal E_1)}{r_1}-  \frac{c_1(\cal E_2)}{r_2}\right)^2\cdot H^{n-2}\\
&\ge \frac{\Delta (\cal E_1)\cdot H^{n-2}}{r_1}+\frac{\Delta (\cal E_2)\cdot H^{n-2}}{r_2}.
\end{align*}
Using  the Bogomolov type inequality \cite[Theorem 3.2]{Langer04} for $\cal E_1$ and $ \cal E_2$, we see that $\Delta (\cal E_1)\cdot H^{n-2}=\Delta (\cal E_2)\cdot H^{n-2}=0$.
Equality in  the Hodge index inequality implies also that $c_1(\cal E_1)$ and $c_1(\cal E_2)$ are numerically trivial.
The induction hypothesis directly applies only to the reflexive sheaf $\cal E_1$ (not the torsion free $\cal E_2$) and we deduce that it is locally free  and numerically flat. But we also have an exact sequence $$0\to \cal E_2\to {\cal E_2}^{\vee\vee}\to Q\to 0,$$ where $Q$ is supported in codimension at least $2$.  Since ${\cal E_2}^{\vee\vee}$ is also strongly $\mu_H$-semistable,
 $\Delta({\cal E_2}^{\vee\vee})\cdot H^{n-2}\geq 0$. But $\Delta({\cal E_2}^{\vee\vee})\cdot H^{n-2}=2r c_2 (Q)\cdot H^{n-2}\le 0$, so
 $\Delta({\cal E_2}^{\vee\vee})\cdot H^{n-2}=0$ and $Q$ is supported in codimension at least $3$. By the induction hypothesis, ${\cal E_2}^{\vee\vee}$ is locally free and every factor in a Jordan--H\"older filtration of ${\cal E_2}^{\vee\vee}$ is also locally free  and numerically flat. An Ext computation using that $Q$ has codimension at least $3$ shows that we have a commutative diagram
	\[\xymatrix{0\ar[r]&\cal E_1\ar@{=}[d]\ar[r]&\cal E\ar[r]\ar[d]& \cal E_2\ar@{^{(}->}[d]\ar[r]&0\\
		0\ar[r]&\cal E_1\ar[r]&\cal E'\ar[r]&{\cal E_2}^{\vee\vee}\ar[r]&0}\]
	for some sheaf $\cal E'$ that is then necessarily locally free. See \cite[Proposition 2.5]{NakayamaNormalized} or \cite[Lemma 1.12]{Langer19} for details. The middle vertical arrow is an isomorphism since $\cal E$ and $\cal E'$ are both reflexive, and isomorphic on the locally free locus of $\cal E_2$. This implies that $\cal E$ is locally free and $\cal E_2$ is reflexive, so we can apply the induction assumption also  to $\cal E_2$.
	
	To finish the proof one needs to deal with the case when $\cal E$ is $\mu_H$-stable
	but not strongly $\mu_H$-stable. Then we can apply the above arguments for some Frobenius pull-back $(F_X^m)^*\cal E$. Since local freeness  and numerical flatness of $(F_X^m)^*\cal E$ implies local freeness  and numerical flatness of $\cal E$, we get the required assertion.
\end{proof}

\begin{cor}\label{local-freeness}
 Let $\cal E$ be a reflexive strongly $\mu_H$-semistable sheaf of rank $r$ with $\Delta(\cal E)\cdot H^{n-2}=0$. Then $\cal E$ is locally free and $\End \cal E$ is numerically flat.  Moreover, every factor in a Jordan--H\"older filtration of $\cal E$ is also locally free and its endomorphism bundle is numerically flat.
\end{cor}

\begin{proof}
	We use a finite cover to extract an $r$-th root of $\det\cal E$ and reduce to the case $c_1(\cal E)\equiv 0$ where Corollary \ref{cor:sstablec1=0} applies. Let $\varphi:Y\to X$ be a Bloch--Gieseker cover (see \cite[Lemma 2.1]{BlochGieseker}), i.e., a finite surjective map from a smooth projective variety $Y$ such that $\varphi^*\det\cal E=\cal L^{\otimes r}$ for some line bundle $\cal L$ on $Y$. 
	The map $\varphi$ is flat so $\varphi^*\cal E$ is reflexive. By Remark \ref{rmk:chernpullback} we also have  $c_i(\varphi^*\cal E)=\varphi^*c_i(\cal E)$ for all $i$. In particular, we have $\det(\varphi^*\cal E\otimes\cal L^{\vee})=\cal O_Y$ and $\Delta(\varphi^*\cal E)\cdot\varphi^*H^{n-2}=0$. Furthermore, $\varphi^*H$ is ample, and $\varphi^*\cal E$ is strongly $\mu_{\varphi^*H}$-semistable. Then $\varphi^*\cal E\otimes\cal L^{-1}$ is reflexive, strongly $\mu_{\phi^*H}$-semistable, with trivial determinant, and $\Delta(\varphi^*\cal E\otimes\cal L^{-1})\cdot\varphi^*H^{n-2}=\Delta(\varphi^*\cal E)\cdot\varphi^*H^{n-2}=0$.
	From the results above we deduce that $\varphi^*\cal E\otimes\cal L^{-1}$ is locally free and numerically flat. Hence   $\varphi^*\End \cal E= \End (\varphi^*\cal E\otimes\cal L^{-1})$ is also numerically flat. This implies that 
 $\cal E$ is locally free and $\End \cal E$ is numerically flat.
 The last part follows analogously. The only difference is that the pull-back of $\mu_H$-stable sheaf need not be $\mu_{\varphi^*H}$-stable and it is only $\mu_{\varphi^*H}$-semistable. So one needs to take a refinement of the pull-back of a Jordan--H\"older filtration of $\cal E$ to 
	a Jordan--H\"older filtration of $\varphi^* \cal E$ and then use Corollary \ref{cor:sstablec1=0}. 
\end{proof}

\subsection{Main theorems in the smooth case}

 \begin{thrm}\label{thm:main3}
	Let $X$ be a smooth projective variety of dimension $n$ defined over an algebraically closed field $k$ and let $H$ be an ample polarization on $X$. Let $\cal E$ be a torsion free sheaf on $X$. Then the following conditions are equivalent:
	\begin{enumerate}
		\item $\cal E$ is reflexive, strongly $\mu_H$-semistable  and $\ch_1(\cal E)\cdot H^{n-1}=\ch_2(\cal E)\cdot H^{n-2}=0$.
		\item $\cal E$ is locally free and numerically flat.
		\item $\cal E$ is strongly $\mu_H$-semistable and $c_j(\cal E)\equiv 0$ for all $j\ge 1$.
		\item $\cal E$ is strongly $\mu_H$-semistable and the normalized Hilbert polynomial of $\cal E$ equals to the Hilbert polynomial of $\cal O_X$.		\end{enumerate}
In particular, if $\cal E$ is a strongly slope semistable vector bundle on $X$ with $c_1(\cal E)\equiv 0$, then $\cal E$ is nef if and only if $\Delta(\cal E)\equiv 0$.
\end{thrm} 

\begin{proof}
	The conditions $\ch_1(\cal E)\cdot H^{n-1}=0$ and $\ch_2(\cal E)\cdot H^{n-2}=0$ imply that $c_1(\cal E)\cdot H^{n-2}$ is numerically trivial and $\Delta(\cal E)\cdot H^{n-2}=0$ by the Bogomolov inequality and by the Hodge index theorem on surfaces (see \cite[Lemma 4.2]{Langerflat}).  The condition $c_1(\cal E)\cdot H^{n-2}=0$ implies that $c_1(\cal E)$ is numerically trivial by \cite[Example 19.3.3]{fulton84}.
	Therefore $(1)\Rightarrow (2)$ follows from Corollary \ref{cor:sstablec1=0}.
	The implication $(2)\Rightarrow (3)$ follows from Proposition \ref{prop:numflat}. The proofs of implications $(3)\Rightarrow(4)$	and $(4)\Rightarrow(1)$  are analogous to the proofs of the corresponding implications in Proposition \ref{thm:main-variation}. 
	
	If $\cal E$ is a nef vector bundle with $c_1(\cal E)\equiv 0$, then $\cal E$ is numerically flat and in particular $\Delta(\cal E)\equiv0$.
	Conversely, if $\cal E$ is a strongly semistable vector bundle with $c_1(\cal E)\equiv 0$, and $\Delta(\cal E)\equiv 0$, then $\ch_2(\cal E)\equiv 0$ and so $\cal E$ is numerically flat.
\end{proof}

Even in the locally free case, one cannot replace condition $(3)$ in Theorem \ref{thm:main3} with $c_j(\cal E)\cdot H^{n-j}=0$ for all $j\geq 1$, or the condition $\ch_2(\cal E)\cdot H^{n-2}=0$ in (1) with $c_2(\cal E)\cdot H^{n-2}=0$.

\begin{ex}
	Let $X$ be a smooth projective surface of Picard rank at least 3. Let $H,L,L'$ be divisors on $X$ with $H$ ample such that 
	the intersection pairing on ${\rm span}(H,L,L')\subseteq N^1(X)$ has diagonal matrix with respect to the basis $(H,L,L')$.
	Let $\cal E=\cal O_X(L)\oplus \cal O_X(L')$. It is strongly $\mu_H$-semistable. Furthermore, we have $c_1(\cal E)\cdot H=0$ and $\int_X c_2(\cal E)=0$. However, $\cal E$ is not numerically flat since $c_1(\cal E)=L+L'$ is not numerically trivial.
	Note that in this case $\int_X\ch_2(\cal E)<0$.
\end{ex}

\begin{thrm}\label{thm:main4}
Let $X$ be a smooth projective variety of dimension $n$ defined over an algebraically closed field $k$ and let $H$ be an ample polarization on $X$.
Let $\cal E$ be a reflexive sheaf of rank $r$ on $X$. Then the following conditions are equivalent:
	\begin{enumerate}
		\item $\cal E$ is  strongly $\mu_H$-semistable and $\Delta(\cal E)$ is numerically trivial.
		\item $\cal E$ is strongly $\mu_H$-semistable and $\Delta(\cal E)\cdot H^{n-2}=0$.
		\item $\cal E$ is locally free and the twisted normalized bundle $\cal E\langle-\frac 1r\det\cal E\rangle$ is nef.
		\item $\cal E$ is locally free and  $\End \cal E$ is nef.
		\item  For every morphism $f:C\to X$ from a smooth projective curve, $f^*\cal E$ is semistable.
		\item For every morphism $f:Y\to X$, where $Y$ is a smooth projective variety, $f^*\cal E$ is strongly slope semistable with respect to any ample polarization on $Y$.
	\end{enumerate} 
	In particular, if $\cal E$ is strongly $\mu_H$-semistable and $\Delta(\cal E)\cdot H^{n-2}=0$, then $\cal E$ is locally free. Furthermore, it is nef (resp.~ample) if and only if $\det\cal E$ is nef (resp.~ample).
\end{thrm}
\begin{proof}
	$(1)\Rightarrow(2)$ is trivial. For locally free $\cal E$, the nefness of $\cal E\langle-\frac 1r\det\cal E\rangle$ is equivalent to the nefness of $\varphi^*\cal E\otimes\cal L^{\vee}$, where $\varphi:Y\to X$ is a finite cover as in Corollary \ref{local-freeness}. Since this bundle has trivial determinant, its nefness is equivalent to the nefness (equivalently numerical flatness) of $\cal E{\rm nd}\,\varphi^*\cal E$ and then to that of $\cal E{\rm nd}\,\cal E$. We get the implications $(2)\Rightarrow(3)\Leftrightarrow(4)$ by Corollary \ref{local-freeness}.
	We also have $(4)\Rightarrow(1)$ by Proposition \ref{prop:numflat}.
	 
	 Numerically flat bundles are universally slope semistable, in fact universally strongly slope semistable. We obtain $(4)\Rightarrow(6)\Rightarrow(5)$.
	 Then non-torsion semistable sheaves on smooth projective curves are torsion free, in particular locally free. General complete intersection curves of high degree passing through a given point $x\in X$ are smooth by \cite{Bertini}. Assuming $(5)$, we obtain that $\cal E$ is locally free.  
	By precomposing $f:C\to X$ with iterates of the Frobenius $F_C$, we see that $(5)$ is equivalent to the analogous statement for strong semistability. On $C$, the strong semistability of $f^*\cal E$ is equivalent to the nefness of $f^*\cal E{\rm nd}\,\cal E$. We deduce that $(5)\Rightarrow(4)$.
	
	For the last statements, if $\cal E$ is strongly $\mu_H$-semistable with $\Delta(\cal E)\cdot H^{n-2}=0$, then $\cal E$ is locally free by Corollary \ref{local-freeness}. Clearly if $\cal E$ is nef (resp.~ample), then $\det\cal E$ is nef (resp.~ample). The implication $(2)\Rightarrow(3)$ and the identity $\cal E=\bigl(\cal E\langle-\frac 1r\det\cal E\rangle\bigr)\langle\frac 1r\det\cal E\rangle$ give the converse.
\end{proof}

\begin{rmk}
In the above theorem one can also give another condition analogous to $(3)$ of Theorem \ref{thm:main3}. See \cite[Theorem 2.2]{Langer19} for the precise formulation. We leave an easy proof of this result along the above lines to the interested reader.
\end{rmk}

\begin{rmk}
	If $X$ is a complex projective manifold then homological equivalence over $\bb Q$ implies numerical equivalence. Classically they are known to agree for divisors. Lieberman \cite{Lieberman} also proved it for codimension 2 cycles using the hard Lefschetz theorem. So in this case proving that $\Delta(\cal E)$ is 0 in $H^4(X,\bb Q)$ is equivalent to proving that it is numerically trivial.
\end{rmk}

\section{On Misra's question}

\begin{defn}
	Let $X$ be a projective variety defined over an algebraically closed field $k$. We say that $X$ is \emph{$1$-homogeneous} if $\Nef(X)=\Eff(X)$.
\end{defn}

\noindent Curves, or more generally varieties of Picard rank 1, and homogeneous spaces are $1$-homogeneous. By \cite[Theorem 3.1 and remark on p.~464]{Miyaoka} 
if $\cal E$ is a vector bundle on  a smooth projective curve then $\bb P(\cal E)$ is $1$-homogeneous if and only if $\cal E$ is strongly semistable.

\begin{rmk}\label{rmk:1st1homogeneous}$ $
	\begin{enumerate}[(i)]
		\item If $X$ is a projective variety with Picard number 2 and $A$ and $B$ are globally generated line bundles, but not big, and not proportional in $N^1(X)$, then $X$ is 1-homogeneous and $A$ and $B$ span the boundary rays of $\Nef(X)=\Eff(X)$.
		
		\item If $X$ is projective with Picard number 1 and dimension $n$, and $\cal E$ is a vector bundle of rank $r$ on $X$ such that $\cal E$ can be generated by fewer than $n+r$ global sections, then $\bb P(\cal E)$ is $1$-homogeneous. (We get an induced morphism $f:\bb P(\cal E)\to\bb P^N$ for some $N<\dim\bb P(\cal E)$. In particular $r\geq 2$ and $\bb P(\cal E)$ has Picard rank 2. The fibers of $\pi:\bb P(\cal E)\to X$ are embedded by $f$. Thus if $H$ is a very ample line bundle on $X$, then $\pi^*H$ and $f^*\cal O_{\bb P^N}(1)$ satisfy the requirements of (i).) 
	\end{enumerate}
\end{rmk}

\subsection{Misra's theorem in an arbitrary characteristic}

The following theorem generalizes \cite[Theorem 1.2]{Misra21} to an arbitrary characteristic:

\begin{thrm}\label{Misra}
	Let $X$ be a smooth projective variety defined over an algebraically closed field $k$ and let $\cal E$ be a strongly slope semistable bundle with respect to some ample polarization of $X$. Let us also assume that $\Delta(\cal E)\equiv 0$.
	Then the following conditions are equivalent:
	\begin{enumerate}
		\item $X$ is $1$-homogeneous, 
		\item $c_1 (\pi _* {\cal O} _{\mathbb P (E)} (D))$ is nef for every effective divisor $D$ on $\bb P(\cal E)$,
		\item $\bb P(\cal E)$ is $1$-homogeneous. 
	\end{enumerate}
\end{thrm}

\begin{proof}
	We specify the adjustments needed to port the proof of \cite[Theorem 1.2]{Misra21} to positive characteristic.
	Once can use Theorem  \ref{thm:main4} instead of 
	\cite[Theorem 2.1]{Misra21}. Apart from that the implication $(1)\Rightarrow (2)$ uses  
	the fact that symmetric powers of $\cal E$ are semistable. This follows either from 
	the Ramanan--Ramanathan theorem (see the proof of Proposition \ref{prop:surfacetocurve}) or one can use Theorem  \ref{thm:main4} and the fact that symmetric powers of nef bundles are nef.
	
	Finally, the proof of $(1)\Rightarrow (2)$ uses  the duality between the cone of strongly  movable curves and pseudo-effective divisors. This fact also holds in positive characteristic by \cite[Theorem 2.22]{fl13z} (see also \cite[Theorem 1.4]{Das}). The rest of the proof is the same as in \cite{Misra21}.
\end{proof}

\subsection{Syzygy bundle counterexamples to Question \ref{ques:main}}

\begin{defn}
	Let $X$ be a projective variety. Let $\cal V$ be a globally generated vector bundle on $X$. The associated \emph{syzygy bundle} $M_{\cal V}$ is the kernel of the natural evaluation morphism $H^0(X,\cal V)\otimes\cal O_X\overset{{\rm ev}}{\twoheadrightarrow}\cal V$. Put $E_{\cal V}=M_{\cal V}^{\vee}$. This is a globally generated vector bundle. 
\end{defn}

For example the Euler sequence on $\bb P^n$ gives $M_{\cal O_{\bb P^n}(1)}=\Omega_{\bb P^n}(1)$. If $\cal L$ is a globally generated line bundle, and $\varphi:X\to\bb P^N$ is the induced morphism with $\varphi^*\cal O_{\bb P^N}(1)=\cal L$, then $M_{\cal L}=\varphi^*\Omega_{\bb P^N}(1)$ and $E_{\cal L}=\varphi^*T_{\bb P^N}(-1)$.

\begin{rmk}\label{rmk:2nd1homogeneous}
	If $\cal V$ is a globally generated vector bundle on $X$ with $\dim X>{\rm rk}\,\cal V$, then $r={\rm rk}\,E_{\cal V}=h^0(X,\cal V)-{\rm rk}\,\cal V$ and $E_{\cal V}$ is generated by $h^0(X,\cal V)=r+{\rm rk}\,\cal V<\dim X+r$ global sections.
	In particular, if $X$ has Picard rank 1 then $\bb P(E_{\cal V})$ is $1$-homogeneous by Remark \ref{rmk:1st1homogeneous}.
\end{rmk}

\begin{prop}\label{prop:easyexample}
	For $n\geq 2$ we have that
	\begin{enumerate}[(i)]
		\item $\bb P(T_{\bb P^n})$ is $1$-homogeneous.
		\item $T_{\bb P^n}$ is strongly slope semistable with respect to the hyperplane class.
		\item $\Delta(T_{\bb P^n})\neq 0$.
		\item The restriction of $T_{\bb P^n}$ to every line in $\bb P^n$ is unstable.
	\end{enumerate} 
\end{prop} 
\begin{proof}
	$(i)$. Apply Remark \ref{rmk:2nd1homogeneous} to $\cal V=\cal O_{\bb P^n}(1)$.	
	$(ii)$ is classical and it follows, e.g., from the Bott vanishing.
	$(iii)$. By direct computation, $\Delta(T_{\bb P^n})=\Delta(T_{\bb P^n}(-1))=n+1$. 
	$(iv)$. $T_{\bb P^n}$ restricts as $\cal O(2)\oplus\cal O(1)^{\oplus n-1}$ on every line.
\end{proof}

The easy counterexample above is slope semistable. We also give a slope unstable example inspired in part by suggestions of S.~Misra and D. S.~Nagaraj.

\begin{ex}\label{ex:unstableexample}
	On $\bb P^3$ consider the globally generated bundle $\cal V=\cal O_{\bb P^3}(1)\oplus\cal O_{\bb P^3}(2)$. Consider the associated {syzygy bundle} $M_{\cal V}$ and let $\cal E=E_{\cal V}=M_{\cal V}^{\vee}$. Then $\cal E$ is slope unstable, has positive discriminant, and $\bb P(\cal E)$ is 1-homogeneous. The bundle $\cal E$ has rank 12, $c_1(\cal E)=3$ and $c_2(\cal E)=7$. It is an immediate computation that $\Delta(\cal E)>0$. We have that $M_{\cal V}=M_{\cal O_{\bb P^3}(1)}\oplus M_{\cal O_{\bb P^3}(2)}$. The summands have slopes $-1/3$ and respectively $-2/9$. Thus $M_{\cal V}$ and its dual $\cal E$ are unstable. Since ${\rm rk}\,\cal V=2<\dim\bb P^3$, we get that $\bb P(\cal E)$ is 1-homogeneous by Remark \ref{rmk:2nd1homogeneous}.\qed
\end{ex}

We list related problems asking if our counterexamples are the simplest/smallest possible.

\begin{ques}
	
	\noindent
	\begin{enumerate}
		\item Does there exist a complex projective manifold $X$ of Picard rank 1 and dimension at least 2 supporting an ample and globally generated line bundle $L$ such that the syzygy bundle $M_L$ is $\mu_L$-unstable, but $\bb P(M_L^{\vee})$ is $1$-homogeneous? \footnote{\cite{SchneiderO} constructs an example on curves. The semistability of syzygy bundles is an active topic of research. We refer to \cite{elstable,ELM13,ButlerConjecture} and the references therein for a history of the problem.}
		\item Does there exist a complex projective surface $X$ supporting a slope unstable $\cal E$ such that $\bb P(\cal E)$ is $1$-homogeneous? 
		\item Are there any $\mu_H$-unstable bundles $\cal E$ with $\Delta(\cal E)\cdot H^{n-2}=0$ such that $\bb P(\cal E)$ is $1$-homogeneous?
	\end{enumerate}
\end{ques}

\subsection{A positive result}

\begin{lem}\label{plethysm}
	Let $V$ be a free module of rank $r$ over a commutative ring $k$.
	Then for any $a, b\ge 1$ there exist
	\begin{enumerate}
		\item  a surjection of ${\rm GL}\, (V)$-modules
		$$\Sym^a(\Sym^{b} V)\twoheadrightarrow \Sym^{ab}V,$$
		\item  an inclusion of ${\rm GL}\, (V)$-modules
		$$(\bigwedge^rV)^{\otimes 2a}\hookrightarrow \Sym^r(\Sym^{2a} V).$$
	\end{enumerate}
Moreover, the composition 
$$(\bigwedge^rV)^{\otimes 2a}\hookrightarrow \Sym^r(\Sym^{2a} V)\to \Sym^{2ra}V$$ 
is zero.
\end{lem}

\noindent Over $\bb C$, assertion $(2)$ is a particular case of the main result of \cite{BCI} which also applies to other even partitions than $\lambda=((2a)^r)$. 

\begin{proof}
	We have a canonical surjection $\bigotimes _{i=1}^a \Sym^{b} V \twoheadrightarrow \Sym^{ab}V$ coming from the symmetric multiplication. 
	By definition we also have a canonical surjection $\bigotimes _{i=1}^a \Sym^{b} V \twoheadrightarrow \Sym^a(\Sym^{b} V)$. Using the universal property of the symmetric product we get an induced map $\Sym^a(\Sym^{b} V)\to \Sym^{ab}V$, which is also surjective.
	This gives the first assertion.
	
	To prove the second assertion we reduce to the case $k=\bb Z$. For $k=\bb Z$ we construct an explicit non-zero map of ${\rm GL}\,(V)$-modules that has an associated matrix with an entry equal to $1$. The map is then a split inclusion as a morphism of $\bb Z$-modules, hence base changing to any commutative ring $k$ is still injective.	
	
	Let $\lambda=((2a)^r)=(2a,...,2a)$ be a partition of $2ar$, i.e., we have a  rectangle of size $(r\times 2a)$. Let $\Sigma$ be the set of all tableaux $T$ of shape $\lambda$  with the entries in $[1,r]=\{1,...,r\}$ so that in each column we have a permutation of the set $[1,r]$ and the first column corresponds to an even  permutation.
	We set $\sgn T=\prod _{i=1}^{2a}\sgn \sigma _i$, where $\sigma _i$ is the permutation of $[1,r]$ corresponding to the $i$-th column of $T$ and $\sgn \sigma$ is the sign of permutation $\sigma$.
	Now we define the map 	
	$$\varphi: \prod _{i=1}^{2a}\left(\prod _{j=1}^r V\right)\to \Sym^r(\Sym^{2a} V)$$
	by setting
	$$\varphi ((v_{11}, ..., v_{r1}), ...,(v_{1,2a}, ..., v_{r,2a}))=\sum _{T\in \Sigma}
	{\sgn T} \prod _{j=1}^r \left( \prod _{i=1}^{2a} v_{T (j, i), i}\right).$$
	Since this map is multilinear it factors to the map
	$$\varphi: (V^{\otimes r})^{\otimes 2a}=\bigotimes _{i=1}^{2a}\left(\bigotimes _{j=1}^r V\right)\to \Sym^r(\Sym^{2a} V)$$
	Note that this last map is alternating in each set of variables $(v_{1m}, ... , v_{rm})$, where $m\in [1, 2a]$. 
	Since we work over $\bb Z$, it is sufficient to check that the corresponding multilinear form is antisymmetric. This is clear for $m>1$ as exchanging $v_{im}$ with $v_{jm}$ defines a bijection on the set $\Sigma$ that replaces the tableau $T$ with another tableau with exchanged entries between $(i,m)$ and $(j,m)$ places. For $m=1$ it follows from the fact that exchanging $v_{i1}$ and $v_{j1}$ defines a bijection on the set $\Sigma$ that replaces 
	the tableau $T$ with another tableau with the same first column but  exchanged $i$-th and $j$-th entries on all of the remaining $2a-1$ columns. This changes the sign with which the corresponding product is taken. Therefore $\varphi$ is antisymmetric also in the variables 
	$(v_{11}, ... , v_{r1})$. This implies that the formula 
	$$\bigotimes_{i=1}^{2a}(v_{1i}\wedge ... \wedge v_{ri})\to \sum _{T\in \Sigma}
	{\sgn T} \prod _{j=1}^r \left( \prod _{i=1}^{2a} v_{T (j, i), i}\right)$$
	defines a map of ${\rm GL}\, (V)$-modules
	$$(\bigwedge^rV)^{\otimes 2a}\to \Sym^r(\Sym^{2a} V).$$
	If $(e_1,...,e_r)$ is a basis of $V$, the element $\bigotimes_{i=1}^{2a}(e_1\wedge ... \wedge e_{r})$ is mapped to 
	$$W=\sum _{T\in \Sigma}{\sgn T}
	\prod _{j=1}^r \left(
	\prod _{i=1}^{2a} e_{T(j,i)}\right).$$
	Note that $\Sym^r(\Sym^{2a} V)$ has a standard basis corresponding to 
	$$\prod _{j=1}^r \left(
	\prod _{i=1}^{r} e_{i}^{n_{ij} } \right), $$
	where $\sum_i n_{ij}=2a$ for $j=1,...,r$.
	If we write $W$ in this basis, the coefficient at the element  $\prod _{j=1}^r e_j^{2a}$ is equal to $1$, so the corresponding map is non-zero.

To see the last part of the lemma, it is sufficient to remark that we have
	$$\sum _{T\in \Sigma}{\sgn T}
\prod _{j=1}^r 
\prod _{i=1}^{2a} e_{T(j,i)}=\sum _{T\in \Sigma}{\sgn T}
\prod _{i=1}^{2a}\prod _{j=1}^r e_{T(j,i)}=(\sum _{T\in \Sigma}{\sgn T})
\prod _{i=1}^{2a}\prod _{j=1}^re_j=0$$
in $ \Sym^{2ra}V$.
\end{proof}

\begin{rmk}
	{\cite[Example 1.9]{Weintraub} shows that the plethysm $\Sym^5(\Sym^3\bb C^5)$ does not contain $(\bigwedge^5\bb C^5)^{\otimes 3}$ as a ${\rm GL}\, (\mathbb C, 5)$-submodule. Thus the parity condition in the above lemma is necessary.}	
\end{rmk}

\begin{cor}\label{plethysm2}
	Let $\cal E$ be a rank $r$ vector bundle on some scheme $X$ defined over some commutative ring $k$. Then  for any $a, b\ge 1$ we have
	\begin{enumerate}
		\item  a canonical surjection 
		$$\Sym^a(\Sym^{b} \cal E)\twoheadrightarrow \Sym^{ab}\cal E,$$
		\item  a canonical inclusion
		$$(\det \cal E)^{\otimes 2a}\hookrightarrow \Sym^r(\Sym^{2a} \cal E)$$
		onto a subbundle.
	\end{enumerate}
Moreover, the composition 
$$(\det \cal E)^{\otimes 2a}\hookrightarrow \Sym^r(\Sym^{2a}  \cal E)\twoheadrightarrow \Sym^{2ra} \cal E$$ 
is zero.
\end{cor}

\begin{proof}
	The corollary follows immediately from the previous lemma. For the convenience of the reader we recall the idea of proof.
	Let $V$ be a free $k$-module of rank $r$ and let $P\to X$  be the principal ${\rm GL \, }(V)$-bundle associated to $\cal E$. Then for any ${\rm GL }\, (V)$-module $W$ we have the associated vector bundle $P(W)$ and maps of ${\rm GL }\, (V)$-modules induce the corresponding maps of vector bundles. Applying this construction to maps from Lemma \ref{plethysm}, we get the corresponding maps from the corollary.
\end{proof} 

\medskip

When $\cal E$ is a $\mu_H$-semistable bundle on $(X,H)$ such that $\Delta(\cal E)\equiv0$, \cite[Lemma 2.3]{Misra21} observes that $\Sym^m\cal E$ is also $\mu_H$-semistable and $\Delta(\Sym^m\cal E)\equiv0$ for all $m\geq 0$. If furthermore $X$ is $1$-homogeneous, then it follows from \cite[Theorem 1.2]{Misra21} that $\bb P(\Sym^m\cal E)$ is $1$-homogeneous for all $m\geq 0$. Question \ref{ques:main} should also consider $m\geq 1$.

\begin{ex}\label{ex:2p2}
	On $\bb P^2$ consider $\cal E=\Sym^2(T_{\bb P^2}(-1))$. Then $\bb P(\cal E)\simeq{\rm Hilb}^2\bb P^2$. The divisor $E$ of nonreduced length 2 subschemes of $\bb P^2$ is contracted by the birational Hilbert--Chow morphism. In particular, it is effective, but not a nef divisor. If $L=\cal O_{\bb P(\cal E)}(1)$ and $H$ is the pullback of the class of a line in $\bb P^2$ then $E$ is linearly equivalent to $2(L-H)$. See \cite[Section 7.2]{fl13z} for details. 
	
	Another perspective at this example is as follows. 
Since $\cal E$ has rank $2$, Corollary \ref{plethysm2} and comparison of ranks imply that we have a short exact sequence of vector bundles
	$$0\to (\det \cal E)^{\otimes 2}\to\Sym^2(\Sym^{2}  \cal E)\to \Sym^{4} \cal E\to 0.$$ 
This gives a short exact sequence 
\[0\to \cal O_{\bb P^2}\to\bigl(\Sym^2(\Sym^2T_{\bb P^2}(-1))\bigr)(-2)\to \bigl(\Sym^4(T_{\bb P^2}(-1))\bigr)(-2)\to 0.\] 
In particular, we see that $\bigl(\Sym^2(\Sym^2T_{\bb P^2}(-1))\bigr)(-2)$ is effective. However, it is not nef, e.g., because its quotient $\bigl(\Sym^4(T_{\bb P^2}(-1))\bigr)(-2)$ restricts to $\cal O_L(-2)\oplus\cal O_L(-1)\oplus\cal O_L\oplus\cal O_L(1)\oplus\cal O_L(2)$ on every line $L$ in $\mathbb P^2$.  	
\end{ex}

\begin{rmk}\label{eq:condition}
	Let  $\cal E$ is a vector bundle on $X$. Then one can easily see that the following conditions are equivalent:
	\begin{enumerate}
		\item 
		$\bb P(\cal E)$ is $1$-homogeneous
		\item  If $D$ is a divisor on $X$ such that $\bigl(\Sym^m\cal E\bigr)(D)$ is effective for some $m\geq 0$ then $\bigl(\Sym^m\cal E\bigr)(D)$ is nef.
	\end{enumerate} 
\end{rmk}

\begin{thrm}
	Let $X$ be a smooth projective variety defined over an algebraically closed field $k$. Let $\cal E$ be a vector bundle of rank $r$ on $X$. Then 
	$\cal E$ is strongly slope semistable with respect to any ample polarization and $\Delta(\cal E)\equiv 0$ if any of the following conditions hold. 
	\begin{enumerate}
		\item For every $m\geq 0$, we have that $\bb P(\Sym^m\cal E)$ is $1$-homogeneous.
		\item For every divisor $D$ on $X$ and every $m,l\geq 0$, we have that $(\Sym^l(\Sym^m\cal E))(D)$ is effective if and only if it is nef.
		\item There exists $m\geq 1$ such that $\Sym^r(\Sym^{2m}\cal E)\otimes(\det\cal E^{\vee})^{\otimes  2m}$ is nef. 
		\item $\bb P(\cal E^{\otimes r})$ is $1$-homogeneous.
		\item $\bb P(\End\cal E)$ is $1$-homogeneous.
	\end{enumerate}
\end{thrm}
\begin{proof}
	The equivalence of $(1)$ and $(2)$ follows from Remark \ref{eq:condition}.
	We focus on $(2)$. By Corollary \ref{plethysm2} for all $m\geq 1$,
	the bundle $\Sym^r(\Sym^{2m}\cal E)$ contains $(\det\cal E)^{\otimes 2m}$.
	Thus $\Sym^r(\Sym^{2m}\cal E)\otimes(\det\cal E^{\vee})^{\otimes 2m}$ is effective and hence it is also nef. Since  Corollary \ref{plethysm2} implies that   $\Sym^{2mr}\cal E
	\otimes(\det\cal E^{\vee})^{\otimes 2m}$ is a quotient of  $\Sym^r(\Sym^{2m}\cal E)\otimes(\det\cal E^{\vee})^{\otimes 2m}$, it is also nef. 
	Since nefness for (twisted) vector bundles is homogeneous (cf.~\cite[Theorem 6.2.12]{laz042}, or \cite[Lemma 3.24 and Remark 3.10]{fm19}), we deduce that $\cal E\langle-\frac 1r\det\cal E^{\vee}\rangle$ is nef. Conclude by Theorem \ref{thm:main}. This argument also handled $(3)$.
	
	$(4)$
	We have a natural inclusion $\det\cal E=\bigwedge^r\cal E \hookrightarrow \cal E^{\otimes r}$. It is obtained by dualizing the natural surjection $(\cal E^{\vee})^{\otimes r}\twoheadrightarrow\bigwedge^r(\cal E^{\vee})$. It shows that  $\cal E^{\otimes r}\otimes\det\cal E^{\vee}$ is effective. By the assumption on the positive cones, it is then also nef. Hence so is its quotient $\Sym^r\cal E\otimes\det\cal E^{\vee}$. Argue as above.
	
	$(5)$ We have a natural inclusion $\cal O_X\hookrightarrow {\cal E}{\rm nd}\, \cal E$ induced by sending $1\in \cal O_X (U)$ to ${\rm id }_{\cal E (U)}$. Therefore 
	${\cal E}{\rm nd}\, \cal E$ is effective and hence
	our assumption implies that it is nef. Now Theorem \ref{thm:main4} implies the required assertion.
\end{proof}

Together with Theorem \ref{Misra} this implies the following result:

\begin{cor}
Let $X$ be a smooth projective variety defined over an algebraically closed field $k$ and let $\cal E$ be a vector bundle of rank $r$ on $X$. Then the following conditions are equivalent:
\begin{enumerate}
\item $\bb P(\Sym^m\cal E)$ is $1$-homogeneous  for every $m\geq 0$.
\item $\bb P(\Sym^{2m}\cal E)$ is $1$-homogeneous  for some $m\geq 1$.
\item $\bb P(\cal E^{\otimes r})$ is $1$-homogeneous.
\item $\bb P(\End\cal E)$ is $1$-homogeneous.
\item The bundle $\End \cal E$ is nef and $X$ is $1$-homogeneous.
\end{enumerate}	
\end{cor}

\bibliographystyle{amsalpha}
\bibliography{References}

\providecommand{\bysame}{\leavevmode\hbox to3em{\hrulefill}\thinspace}
\providecommand{\MR}{\relax\ifhmode\unskip\space\fi MR }
\providecommand{\MRhref}[2]{%
  \href{http://www.ams.org/mathscinet-getitem?mr=#1}{#2}
}
\providecommand{\href}[2]{#2}
\begin{thebibliography}{BPMGNO19}

\bibitem[Bar71]{Barton71}
Charles~M. Barton, \emph{Tensor products of ample vector bundles in
  characteristic {$p$}}, Amer. J. Math. \textbf{93} (1971), 429--438.
  \MR{289525}

\bibitem[BB08]{BBprincipal}
Indranil Biswas and Ugo Bruzzo, \emph{On semistable principal bundles over a
  complex projective manifold}, Int. Math. Res. Not. IMRN (2008), no.~12, Art.
  ID rnn035, 28. \MR{2426752}

\bibitem[BBG19]{BBG19}
Indranil Biswas, Ugo Bruzzo, and Sudarshan Gurjar, \emph{Higgs bundles and
  fundamental group schemes}, Adv. Geom. \textbf{19} (2019), no.~3, 381--388.
  \MR{3982575}

\bibitem[BCI11]{BCI}
Peter B\"{u}rgisser, Matthias Christandl, and Christian Ikenmeyer, \emph{Even
  partitions in plethysms}, J. Algebra \textbf{328} (2011), 322--329.
  \MR{2745569}

\bibitem[BG71]{BlochGieseker}
Spencer Bloch and David Gieseker, \emph{The positivity of the {C}hern classes
  of an ample vector bundle}, Invent. Math. \textbf{12} (1971), 112--117.
  \MR{297773}

\bibitem[BHP14]{BHP14}
Indranil Biswas, Amit Hogadi, and A.~J. Parameswaran, \emph{Pseudo-effective
  cone of {G}rassmann bundles over a curve}, Geom. Dedicata \textbf{172}
  (2014), 69--77. \MR{3253771}

\bibitem[BHR06]{BHHiggs}
U.~Bruzzo and D.~Hern\'{a}ndez~Ruip\'{e}rez, \emph{Semistability vs. nefness
  for ({H}iggs) vector bundles}, Differential Geom. Appl. \textbf{24} (2006),
  no.~4, 403--416. \MR{2231055}

\bibitem[BPMGNO19]{ButlerConjecture}
L.~Brambila-Paz, O.~Mata-Guti\'{e}rrez, P.~E. Newstead, and Angela Ortega,
  \emph{Generated coherent systems and a conjecture of {D}. {C}. {B}utler},
  Internat. J. Math. \textbf{30} (2019), no.~5, 1950024, 25. \MR{3961440}

\bibitem[Das20]{Das}
Omprokash Das, \emph{Finiteness of log minimal models and nef curves on 3-folds
  in characteristic {$p>5$}}, Nagoya Math. J. \textbf{239} (2020), 76--109.
  \MR{4138896}

\bibitem[DH91]{Bertini}
Steven Diaz and David Harbater, \emph{Strong {B}ertini theorems}, Trans. Amer.
  Math. Soc. \textbf{324} (1991), no.~1, 73--86. \MR{986689}

\bibitem[dJ96]{dejong96}
Aise~Johan de~Jong, \emph{Smoothness, semi-stability and alterations}, Inst.
  Hautes \'Etudes Sci. Publ. Math. (1996), no.~83, 51--93.

\bibitem[DPS94]{DPS94}
Jean-Pierre Demailly, Thomas Peternell, and Michael Schneider, \emph{Compact
  complex manifolds with numerically effective tangent bundles}, J. Algebraic
  Geom. \textbf{3} (1994), no.~2, 295--345. \MR{1257325}

\bibitem[DW20]{De-We}
Christopher Deninger and Annette Werner, \emph{Parallel transport for vector
  bundles on {$p$}-adic varieties}, J. Algebraic Geom. \textbf{29} (2020),
  no.~1, 1--52. \MR{4028065}

\bibitem[EL92]{elstable}
Lawrence Ein and Robert Lazarsfeld, \emph{Stability and restrictions of
  {P}icard bundles, with an application to the normal bundles of elliptic
  curves}, Complex projective geometry ({T}rieste, 1989/{B}ergen, 1989), London
  Math. Soc. Lecture Note Ser., vol. 179, Cambridge Univ. Press, Cambridge,
  1992, pp.~149--156. \MR{1201380}

\bibitem[ELM13]{ELM13}
Lawrence Ein, Robert Lazarsfeld, and Yusuf Mustopa, \emph{Stability of syzygy
  bundles on an algebraic surface}, Math. Res. Lett. \textbf{20} (2013), no.~1,
  73--80. \MR{3126723}

\bibitem[FL83]{fl83}
William Fulton and Robert Lazarsfeld, \emph{Positive polynomials for ample
  vector bundles}, Ann. of Math. (2) \textbf{118} (1983), no.~1, 35--60.

\bibitem[FL17a]{flpos}
Mihai Fulger and Brian Lehmann, \emph{Positive cones of dual cycle classes},
  Algebraic Geometry \textbf{4} (2017), no.~1, 1--28.

\bibitem[FL17b]{fl13z}
\bysame, \emph{Zariski decompositions of numerical cycle classes}, J. Algebraic
  Geom. \textbf{26} (2017), no.~1, 43--106. \MR{3570583}

\bibitem[FM21]{fm19}
Mihai Fulger and Takumi Murayama, \emph{Seshadri constants for vector bundles},
  J. Pure Appl. Algebra \textbf{225} (2021), no.~4, 106559, 35. \MR{4158762}

\bibitem[Ful98]{fulton84}
William Fulton, \emph{Intersection theory}, second ed., Ergeb. Math. Grenzgeb.
  (3), vol.~2, Springer-Verlag, Berlin, 1998. \MR{1644323}

\bibitem[Ful11]{ful11}
Mihai Fulger, \emph{Cones of effective cycles on projective bundles over
  curves}, Math. Z. \textbf{269} (2011), no.~1-2, 449--459. \MR{2836078}

\bibitem[Ful20]{FulgerConesvect}
\bysame, \emph{Cones of positive vector bundles}, Rev. Roumaine Math. Pures
  Appl. \textbf{65} (2020), no.~3, 285--302. \MR{4216530}

\bibitem[Gie71]{Gieseker}
David Gieseker, \emph{{$p$}-ample bundles and their {C}hern classes}, Nagoya
  Math. J. \textbf{43} (1971), 91--116. \MR{296078}

\bibitem[GKP16]{GKP16}
Daniel Greb, Stefan Kebekus, and Thomas Peternell, \emph{Movable curves and
  semistable sheaves}, Int. Math. Res. Not. IMRN (2016), no.~2, 536--570.
  \MR{3493425}

\bibitem[Har70]{Har70}
Robin Hartshorne, \emph{Ample subvarieties of algebraic varieties}, Lecture
  Notes in Mathematics, Vol. 156, Springer-Verlag, Berlin-New York, 1970, Notes
  written in collaboration with C. Musili.

\bibitem[Har71]{har71}
\bysame, \emph{Ample vector bundles on curves}, Nagoya Math. J. \textbf{43}
  (1971), 73--89. \MR{292847}

\bibitem[HL10]{HL10}
Daniel Huybrechts and Manfred Lehn, \emph{The geometry of moduli spaces of
  sheaves}, second ed., Cambridge Mathematical Library, Cambridge University
  Press, Cambridge, 2010. \MR{2665168}

\bibitem[JP15]{Joshi-Pauly}
Kirti Joshi and Christian Pauly, \emph{Hitchin-{M}ochizuki morphism, opers and
  {F}robenius-destabilized vector bundles over curves}, Adv. Math. \textbf{274}
  (2015), 39--75. \MR{3318144}

\bibitem[Kle69]{Kleiman69}
Steven~L. Kleiman, \emph{Ample vector bundles on algebraic surfaces}, Proc.
  Amer. Math. Soc. \textbf{21} (1969), 673--676. \MR{251044}

\bibitem[Kle05]{Kleiman-Picard}
\bysame, \emph{The {P}icard scheme}, Fundamental algebraic geometry, Math.
  Surveys Monogr., vol. 123, Amer. Math. Soc., Providence, RI, 2005,
  pp.~235--321. \MR{2223410}

\bibitem[Kol96]{kollarrational}
J\'anos Koll\'ar, \emph{Rational curves on algebraic varieties}, Ergeb. Math.
  Grenzgeb. (3), vol.~32, Springer-Verlag, Berlin, 1996. \MR{1440180}

\bibitem[Lan04]{Langer04}
Adrian Langer, \emph{Semistable sheaves in positive characteristic}, Ann. of
  Math. (2) \textbf{159} (2004), no.~1, 251--276. \MR{2051393}

\bibitem[Lan11]{Langerflat}
\bysame, \emph{On the {S}-fundamental group scheme}, Ann. Inst. Fourier
  (Grenoble) \textbf{61} (2011), no.~5, 2077--2119 (2012). \MR{2961849}

\bibitem[Lan12]{Langerflat2}
\bysame, \emph{On the {S}-fundamental group scheme. {II}}, J. Inst. Math.
  Jussieu \textbf{11} (2012), no.~4, 835--854. \MR{2979824}

\bibitem[Lan15]{Langer15}
\bysame, \emph{Bogomolov's inequality for {H}iggs sheaves in positive
  characteristic}, Invent. Math. \textbf{199} (2015), no.~3, 889--920.
  \MR{3314517}

\bibitem[Lan19]{Langer19}
\bysame, \emph{Nearby cycles and semipositivity in positive characteristic},
  2019, to appear in \emph{J. Eur. Math. Soc.}, arXiv:1902.05745v3 [math.AG].

\bibitem[Lan21]{La-Chern}
\bysame, \emph{On algebraic chern classes of flat vector bundles}, 2021,
  arXiv:2107.03127 [math.AG].

\bibitem[Laz04a]{laz04}
Robert Lazarsfeld, \emph{Positivity in algebraic geometry. {I}}, Ergeb. Math.
  Grenzgeb. (3), vol.~48, Springer-Verlag, Berlin, 2004, Classical setting:
  line bundles and linear series. \MR{2095472}

\bibitem[Laz04b]{laz042}
\bysame, \emph{Positivity in algebraic geometry. {II}}, Ergeb. Math. Grenzgeb.
  (3), vol.~49, Springer-Verlag, Berlin, 2004, Positivity for vector bundles,
  and multiplier ideals. \MR{2095472}

\bibitem[Lie68]{Lieberman}
David~I. Lieberman, \emph{Numerical and homological equivalence of algebraic
  cycles on {H}odge manifolds}, Amer. J. Math. \textbf{90} (1968), 366--374.
  \MR{230336}

\bibitem[LP97]{LePotier97}
J.~Le~Potier, \emph{Lectures on vector bundles}, Cambridge Studies in Advanced
  Mathematics, vol.~54, Cambridge University Press, Cambridge, 1997, Translated
  by A. Maciocia. \MR{1428426}

\bibitem[Mis21]{Misra21}
Snehajit Misra, \emph{Pseudo-effective cones of projective bundles and weak
  {Z}ariski decomposition}, Eur. J. Math. \textbf{7} (2021), no.~4, 1438--1457.
  \MR{4340943}

\bibitem[Miy87]{Miyaoka}
Yoichi Miyaoka, \emph{The {C}hern classes and {K}odaira dimension of a minimal
  variety}, Algebraic geometry, {S}endai, 1985, Adv. Stud. Pure Math., vol.~10,
  North-Holland, Amsterdam, 1987, pp.~449--476. \MR{946247}

\bibitem[Mor98]{Moriwaki98}
Atsushi Moriwaki, \emph{Relative {B}ogomolov's inequality and the cone of
  positive divisors on the moduli space of stable curves}, J. Amer. Math. Soc.
  \textbf{11} (1998), no.~3, 569--600. \MR{1488349}

\bibitem[MR21]{MisraRayAmple}
Snehajit Misra and Nabanita Ray, \emph{On {A}mpleness of vector bundles}, C. R.
  Math. Acad. Sci. Paris \textbf{359} (2021), 763--772. \MR{4311802}

\bibitem[MR82]{MehtaRamanathan}
V.~B. Mehta and A.~Ramanathan, \emph{Semistable sheaves on projective varieties
  and their restriction to curves}, Math. Ann. \textbf{258} (1981/82), no.~3,
  213--224. \MR{649194}

\bibitem[Nak99]{NakayamaNormalized}
Noboru Nakayama, \emph{Normalized tautological divisors of semi-stable vector
  bundles}, no. 1078, 1999, Free resolutions of coordinate rings of projective
  varieties and related topics (Japanese) (Kyoto, 1998), pp.~167--173.
  \MR{1715587}

\bibitem[RR84]{RR84}
S.~Ramanan and A.~Ramanathan, \emph{Some remarks on the instability flag},
  Tohoku Math. J. (2) \textbf{36} (1984), no.~2, 269--291. \MR{742599}

\bibitem[Sch05]{SchneiderO}
Olivier Schneider, \emph{Stabilit\'{e} des fibr\'{e}s {$\Lambda^p E_L$} et
  condition de {R}aynaud}, Ann. Fac. Sci. Toulouse Math. (6) \textbf{14}
  (2005), no.~3, 515--525. \MR{2172589}

\bibitem[Sim92]{Simpson}
Carlos~T. Simpson, \emph{Higgs bundles and local systems}, Inst. Hautes
  \'{E}tudes Sci. Publ. Math. (1992), no.~75, 5--95. \MR{1179076}

\bibitem[Wei90]{Weintraub}
Steven~H. Weintraub, \emph{Some observations on plethysms}, J. Algebra
  \textbf{129} (1990), no.~1, 103--114. \MR{1037395}

\end{thebibliography}

\end{document}